\documentclass[twoside,a4paper]{amsart}

\usepackage[colorlinks=true, breaklinks=true, urlcolor=webbrown, linkcolor=RoyalBlue, citecolor=webgreen, backref=page]{hyperref}
\usepackage[
nochapters, 
beramono, 
eulermath,
pdfspacing, 
dottedtoc 
]{classicthesis1}
\usepackage{amsmath,amssymb,amsthm}
\usepackage{epsfig, graphicx}
\usepackage{verbatim,setspace}

\rohead{\mbox{\color{halfgray} \rightmark\hfil \hspace{0.5em} \rlap{\small \vline \kern1em\color{black}\thepage}}}
\lehead{\mbox{\llap{\small\thepage\kern1em\color{halfgray} \vline}\color{halfgray}\hspace{0.5em}\rightmark\hfil}} 
\newcommand{\R}		{\mathbb{R}}
\newcommand{\CP}		{\mathbb{C}}
\newcommand{\N}		{\mathbb{N}}

\newcommand{\cws}{\stackrel{*}{\to}}

\newcommand{\diag}{\mathsf{diag}}
\newcommand{\supp}{\mathsf{supp}}

\newcommand{\re}{\mathsf{Re}}

\renewcommand{\det}{\mathsf{det}}
\renewcommand{\deg}{\mathsf{deg}}

\newcommand{\RS}{\boldsymbol{\mathfrak{R}}}

\newcommand{\z}	{{\boldsymbol z}}

\newcommand{\x}	{{\boldsymbol x}}

\newcommand{\rhy}   {\textnormal{RHP}-${\boldsymbol Y}$}
\newcommand{\rhx}   {\textnormal{RHP}-${\boldsymbol X}$}

\newcommand{\rhz}   {\textnormal{RHP}-${\boldsymbol Z}$}
\newcommand{\rhn}   {\textnormal{RHP}-${\boldsymbol N}$}

\newtheorem{theorem}{Theorem}
\newtheorem{proposition}[theorem]{Proposition}

\newtheorem{corollary}[theorem]{Corollary}
\newtheorem{lemma}[theorem]{Lemma}
\newtheorem*{definition}{Definition}
\theoremstyle{remark}
\newtheorem*{remark}{Remark}

\begin{document}

\title[Szeg\H{o}-Type Asymptotics of Frobenius-Pad\'e approximants]{Szeg\H{o}-type asymptotics  for ray sequences of Frobenius-Pad\'e approximants}

\author[A.I. Aptekarev]{Alexander I. Aptekarev}

\address{Keldysh Institute of Applied Mathematics, Russian Academy of Science, Moscow, 125047, Russia}

\email{\href{mailto:aptekaa@keldysh.ru}{aptekaa@keldysh.ru}}

\author[A.I. Bogolubsky]{Alexey I. Bogolubsky}

\address{Pirogov Russian National Research Medical University, Ostrovitianov str.~1, Moscow, 117997, Russia}

\email{\href{mailto:bogolub@gmail.com}{bogolub@gmail.com}}

\author[M. Yattselev]{Maxim L. Yattselev}

\address{Department of Mathematical Sciences, Indiana University-Purdue University Indianapolis, 402~North Blackford Street, Indianapolis, IN 46202}

\email{\href{mailto:maxyatts@iupui.edu}{maxyatts@iupui.edu}}

\thanks{The research of the first author (A.I. Aptekarev) was supported by a grant of the Russian Science Foundation, RScF-14-21-00025. The research of the second author (A.I. Bogolubsky) was supported by  grants of the scientific school NSh--8033.2010.1, and RFBR--14-01-00604. The research of the third author (M.L. Yattselev) was supported by a grant from the Simons Foundation, \#354538.}

\subjclass[2000]{}


\maketitle

\begin{abstract}

Let $\widehat\sigma$ be a Cauchy transform of a possibly complex-valued Borel measure $\sigma$ and $\{p_n\}$ be a system of orthonormal polynomials with respect to a measure $\mu$, $\supp(\mu)\cap\supp(\sigma)=\varnothing$. An $(m,n)$-th Frobenius-Pad\'e approximant to $\widehat\sigma$ is a rational function $P/Q$, $\deg(P)\leq m$, $\deg(Q)\leq n$, such that the first $m+n+1$ Fourier coefficients of the linear form $Q\widehat\sigma-P$ vanish when the form is developed into a series with respect to the polynomials $p_n$. We investigate the convergence of the Frobenius-Pad\'e approximants to $\widehat\sigma$ along ray sequences $\frac n{n+m+1}\to c>0$, $n-1\leq m$, when $\mu$ and $\sigma$ are supported on intervals on the real line and their Radon-Nikodym derivatives with respect to the arcsine distribution of the respective interval are holomorphic functions.

\end{abstract}

\maketitle

\section{Introduction}
\label{sec:intro}

Representation of  functions by means of \textit{series  with respect to the Chebyshev polynomials} is a very convenient tool in numerical analysis (see, for example, \cite{ChebFun14}). Such a series converges in the interior of the largest ellipse into which the function has holomorphic continuation. However, if we need to compute the function beyond the boundary of the maximal ellipse of convergence of the series with respect to the Chebyshev polynomials (or any other orthonormal polynomial sequences) then one has to employ rational rather then polynomial approximation of the orthogonal polynomial expansion (see \cite{Clenshaw74}, \cite{Gragg77}). The construction of these rational approximants is related to the notion of the generalized \textit{Pad\'e  table}  \cite{Baker81}. We call them \textit{Pad\'e approximants of an orthogonal expansion}. In this paper we focus on the \textit{Frobenius-Pad\'e approximants}, which are defined by means of a linear system with constant coefficients which are precisely the coefficients of the polynomial expansion of the approximated function (see \eqref{Frobenius}, below). This type of approximants is the most popular in practice due to the ease of their numerical computation.

Let $\mu$ be a possibly complex-valued Borel measure supported on an interval $\Delta_\mu\subset\R$. Assume that $\mu$ possesses full \emph{orthonormal} system of polynomials that we denote by $\{p_n\}$, i.e.,
\[
\int p_n(x)p_m(x)\mathrm d\mu(x) = \delta_{nm},
\]
where $\delta_{nm}$ is the usual Kronecker symbol. When $\mu$ is a positive measure such a system always exists. For complex measures \emph{orthogonal polynomials of minimal degree} uniquely exist as well, but it might happen that $\deg(p_n)< n$, in which case $p_n$ is orthogonal to itself and therefore cannot be orthonormalized. Given a function $f\in L^1(\mu)$, we can associate to $f$ a series
\begin{equation}
\label{expansion}
\sum_{i=0}^\infty c_i(f;\mu)p_i(x), \quad c_i(f;\mu) := \int f(x)p_i(x)\mathrm d\mu(x).
\end{equation}

\begin{definition}
A Frobenius-Pad\'e approximant of type $(m,n)\in\N^2$ to $f\in L^1(\mu)$ is a rational function $P_{m,n}/Q_{m,n}$, $\deg(P_{m,n})\leq m$, $\deg(Q_{m,n})\leq n$, such that
\begin{equation}
\label{Frobenius}
c_i(Q_{m,n}f - P_{m,n};\mu) = 0, \quad i\in\{0,\ldots,m+n\}.
\end{equation}
\end{definition}

Frobenius-Pad\'e approximants always exist as finding $Q_{m,n}$ amounts to solving a linear system
\[
\left(\begin{matrix}
c_{m+1}(p_0f;\mu) & \cdots & c_{m+1}(p_nf;\mu)  \\
\vdots & \ddots & \vdots \\
c_{m+n}(p_0f;\mu) & \cdots & c_{m+n}(p_nf;\mu)
\end{matrix}\right)
\left(\begin{matrix}
a_0 \\ \vdots \\ a_n
\end{matrix}\right) =
\left(\begin{matrix}
0 \\ \vdots \\ 0
\end{matrix}\right)
\]
and letting $Q_{m,n}(x)=\sum_{j=0}^n a_jp_j(x)$ (the system has $n$ equations and $n+1$ unknowns), while $P_{m,n}(x)=\sum_{j=0}^m b_jp_j(x)$ uniquely depends on $Q_{m,n}$ via
\[
\left(\begin{matrix}
c_0(p_0f;\mu) & \cdots & c_0(p_nf;\mu)  \\
\vdots & \ddots & \vdots \\
c_m(p_0f;\mu) & \cdots & c_m(p_nf;\mu)
\end{matrix}\right)
\left(\begin{matrix}
a_0 \\ \vdots \\ a_n
\end{matrix}\right) =
\left(\begin{matrix}
b_0 \\ \vdots \\ b_m
\end{matrix}\right).
\]
An approximant may not be unique, however, the one corresponding to $Q_{m,n}$ of the smallest degree is. Hence, if $\deg(Q_{m,n})=n$ for all solutions, the approximant is unique.

The main motivation for using  Pad\'e approximants of orthogonal expansions is due to their convergence in wider domains than the convergence domains of orthogonal expansions themselves. The problems of convergence of the rows of corresponding tables of the Pad\'e approximants of orthogonal expansions have been investigated by S.P.~Suetin \cite{SU78}, \cite{SU79}. The weak asymptotics and the convergence of the \textit{diagonal} (i.e. type $(n-1,n)$) Pad\'e approximants of orthogonal expansions  for Cauchy transforms
\begin{equation}
\label{markov}
\widehat \sigma(z) := \int\frac{\mathrm d\sigma(t)}{t-z},\qquad \quad\sigma ' (t)\,>\,0,\quad t\in \Delta_\sigma\subset\R, \quad \Delta_\mu\cap\Delta_\sigma=\varnothing,
\end{equation}
have been obtained A.A.~Gonchar, E.A.~Rakhmanov and S.P.~Suetin in \cite{GRakhSu91}, \cite{GRakhSu92}.

In this paper we investigate the \textit{strong asymptotics} and convergence properties of the \textit{ray sequences }(i.e. type $(m,n)$ : $n-1\leq m$ and $n/(n+m)\to c>0$) of Frobenius-Pad\'e approximants for Cauchy transforms \eqref{markov} where $\sigma$ is, generally speaking, a complex-valued Borel measure. To motivate the forthcoming definitions, let us (following \cite{GRakhSu91}, \cite{GRakhSu92}) first heuristically describe the asymptotic behavior of the approximants using the formalism of orthogonal polynomials and potential theory.

For the moment, assume that the measures of $\mu$ and $\sigma$ are positive. In this case the linear form
\[
R_{m,n} := Q_{m,n}\widehat\sigma - P_{m,n}
\]
is real-valued on $\Delta_\mu$ and is orthogonal to all polynomials of degree at most $m+n$ with respect to $\mu$ by \eqref{Frobenius}. Therefore it must have at least $m+n+1$ zeros there. Denote by $V_{m,n}$ the monic polynomial whose zeros are the zeros of $R_{m,n}$ on $\Delta_\mu$, $\deg(V_{m,n})\geq m+n+1$. The expression $z^k R_{m,n}(z)/V_{m,n}(z)$, $k\leq\min\{n-1,m\}$, is holomorphic off $\Delta_\sigma$ and is vanishing at infinity with order at least $2$. Then it follows from Cauchy's theorem, Cauchy's integral formula, and \eqref{markov} that
\begin{equation}
\label{ortho}
\int \frac{x^kQ_{m,n}(x)}{V_{m,n}(x)}\mathrm d\sigma(x) = 0, \quad k\leq \min\{n-1,m\}.
\end{equation}
When $n-1\leq m$, the number of orthogonality conditions above is equal to $n$ and therefore $Q_{m,n}$ must have degree $n$ since $\mathrm d\sigma(x)/V_{m,n}(x)$ is a real measure of constant sign on $\Delta_\sigma$. In particular, this implies uniqueness of $Q_{m,n}$ up to a multiplicative factor. On the other hand, similarly to \eqref{ortho}, Cauchy integral formula, \eqref{markov}, and orthogonality of $R_{m,n}$ with respect to $\mu$ yield that
\begin{equation}
\label{ortho-V}
\int \frac{x^kV_{m,n}(x)}{Q_{m,n}(x)} \left( \int \frac{Q_{m,n}^2(t)}{V_{m,n}(t)}\frac{\mathrm d\sigma(t)}{t-x} \right)\mathrm d\mu(x)=0, \quad k\in\{0,\ldots,m+n\}.
\end{equation}
Given mutual orthogonality relations \eqref{ortho} and \eqref{ortho-V}, it is well understood \cite{Nik86,GRakhS97,ApLy10} which measures describe the limiting behavior of the zeros of $Q_{m,n}$ and $V_{m,n}$. Assuming that $n-1\leq m$ and $n/(n+m)\to c>0$, let $\tau_{\mu,c}$, $|\tau_{\mu,c}|=1$, and $\tau_{\sigma,c}$, $|\tau_{\sigma,c}|=c$, be weak$^*$ limit points of the counting measures of the zeros of $V_{m,n}$ and $Q_{m,n}$, respectively, normalized by $n+m$. Then the pair $(\tau_{\mu,c},\tau_{\sigma,c})$ can be uniquely identified as follows \cite{GRakhS97,ApLy10}.

\begin{proposition}
\label{prop:nikishin}
Given $c\in(0,1/2]$, denote by $\mathcal M_c$ the following class of pairs of Borel measures:
\[
\mathcal M_c := \big\{(\tau_\mu,\tau_\sigma):\;\supp(\tau_\nu)\subseteq\Delta_\nu,\;\nu\in\{\mu,\sigma\},\;|\tau_\mu|=1,\;|\tau_\sigma|=c\big\}.
\]
There exists a pair $(\tau_{\mu,c},\tau_{\sigma,c})\in\mathcal M_c$ such that\footnote{In what follows, $V^\nu(z)=-\int\log|z-w|\mathrm d\nu(w)$ is the logarithmic potential of the measure $\nu$.}
\begin{equation}
\label{nikishin}
\left\{
\begin{array}{lll}
2V^{\tau_{\sigma,c}} - V^{\tau_{\mu,c}} = \min_{\Delta_\sigma} (2V^{\tau_{\sigma,c}} - V^{\tau_{\mu,c}})= 3\ell_{\sigma,c} & \text{on} & \supp(\tau_{\sigma,c}), \medskip \\
2V^{\tau_{\mu,c}} - V^{\tau_{\sigma,c}} = \min_{\Delta_\mu} (2V^{\tau_{\mu,c}} - V^{\tau_{\sigma,c}})= 3\ell_{\mu,c} & \text{on} & \supp(\tau_{\mu,c}),
\end{array}
\right.
\end{equation}
for some constants $\ell_{\mu,c}$ and $\ell_{\sigma,c}$. Moreover, if for some pair $(\tau_\mu,\tau_\sigma)\in\mathcal M_c$ relations analogous to \eqref{nikishin} are satisfied, then $(\tau_\mu,\tau_\sigma)=(\tau_{\mu,c},\tau_{\sigma,c})$.  In addition, it holds that
\[
\supp(\tau_{\mu,c})=\Delta_\mu \quad \text{and} \quad \supp(\tau_{\sigma,c})=:\Delta_{\sigma,c} \quad \text{is an interval.}
\]
Furthermore, set
\begin{equation}
\label{domains}
\left\{
\begin{array}{lll}
D_{\sigma,c}^+ & := & \big\{z:\;V^{\tau_{\mu,c}}(z) - 2V^{\tau_{\sigma,c}}(z) + 3\ell_{\sigma,c}>0\big\}, \medskip \\
D_{\sigma,c}^ - & := & \big\{z:\;V^{\tau_{\mu,c}}(z) - 2V^{\tau_{\sigma,c}}(z) + 3\ell_{\sigma,c}<0\big\}.
\end{array}
\right.
\end{equation}
Then $D_{\sigma,c}^+\neq\varnothing$ and $\Delta_{\sigma,c}\subseteq\partial D_{\sigma,c}^+$, $D_{\sigma,c_2}^-\subseteq D_{\sigma,c_1}^-$ when $c_1\leq c_2$, $D_{\sigma,c}^-=\varnothing$ when $c=\frac12$ and $\infty\in D_{\sigma,c}^-\neq\varnothing$ otherwise, and $(\Delta_\sigma\setminus\Delta_{\sigma,c})\subset D_{\sigma,c}^-$, see Figure~\ref{fig:domains}.
\end{proposition}

The domains $D_{\sigma,c}^+$ and $D_{\sigma,c}^-$ are significant for our analysis as we shall prove that the approximants do converge to $\widehat\sigma$ in $D_{\sigma,c}^+$ and diverge to infinity in $D_{\sigma,c}^-$.

\begin{figure}[ht!]
\centering
\includegraphics[scale=.5]{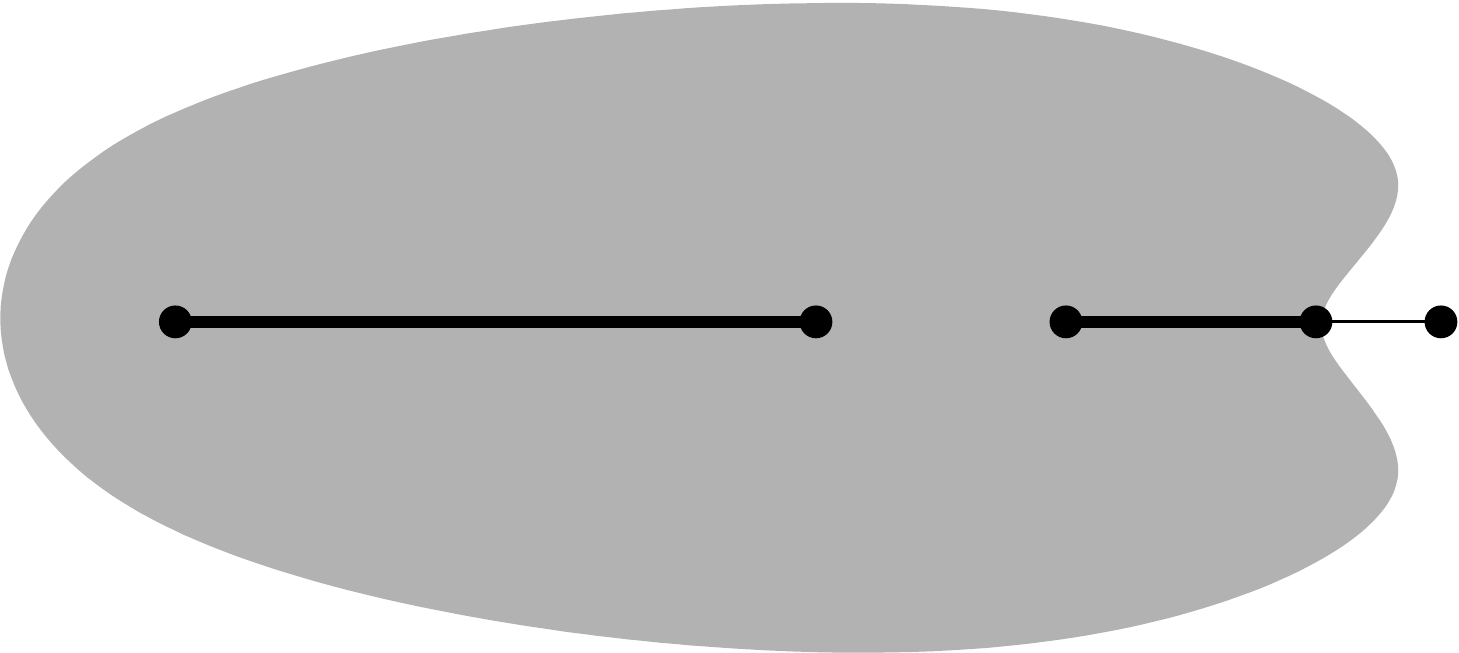}
\begin{picture}(0,0)
\put(-100,70){$D_{\sigma,c}^+$}
\put(-140,53){$\Delta_\mu$}
\put(-50,53){$\Delta_{\sigma,c}$}
\put(-192,38){$b_\mu$}
\put(-100,38){$a_\mu$}
\put(-80,38){$a_\sigma=a_{\sigma,c}$}
\put(-30,38){$b_{\sigma,c}$}
\put(-10,38){$b_\sigma$}
\end{picture}
\caption{\small Intervals $\Delta_\mu=[b_\mu,a_\mu]$, $\Delta_\sigma=[a_\sigma,b_\sigma]$, and $\Delta_{\sigma,c}=[a_{\sigma,c},b_{\sigma,c}]$ and the domain $D_{\sigma,c}^+$ (shaded region).}
\label{fig:domains}
\end{figure}
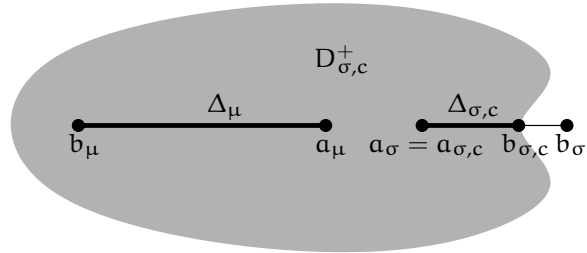

As shown in \cite{Grakh84}, one can describe the weak asymptotics of the polynomials $Q_{m,n}$ and $V_{m,n}$ using the logarithmic potentials of the measures $\tau_{\sigma,c}$ and $\tau_{\mu,c}$. As we aim at strong (Szeg\H{o}) asymptotics we shall omit such a description, which was addressed in \cite{GRakhSu92} for the diagonal case $m=n-1$. Let us point out that relations \eqref{nikishin} are stated differently in \cite{GRakhSu92}. There, see \cite[Equation (2.1)]{GRakhSu92}, it shown that there exists a unique probability measure $\lambda$, $\supp(\lambda)=\Delta_\mu$, and a constant $w$ such that
\[
G^\lambda-3V^\lambda = w \quad \text{on} \quad \Delta_\mu,
\]
where $G^\lambda$ is the Green potential of $\lambda$ relative to $\overline{\mathbb{C}}\setminus\Delta_\sigma$. To rewrite the above relation as system \eqref{nikishin}, recall that $G^\lambda=0$ on $\Delta_\sigma$ and that $G^\lambda=V^{\lambda-\hat\lambda}-\hat w$ in $\mathbb C$ where $\hat w$ is some constant and $\hat\lambda$ is the \emph{balayage} of $\lambda$ onto $\Delta_\sigma$, see \cite[Theorem~II.5.1]{SaffTotik}. Therefore,
\[
\left\{
\begin{array}{lll}
2V^{\hat\lambda/2} - V^{\lambda} = \hat w & \text{on} & \Delta_\sigma, \medskip \\
2V^{\lambda} - V^{\hat\lambda/2} = (w+\hat w)/2 & \text{on} & \Delta_\mu.
\end{array}
\right.
\]
The last equations clearly show that $\tau_{\mu,1/2}=\lambda$ and $\tau_{\sigma,1/2}=\hat\lambda/2$.

\section{Main Results}
\label{sec:main}

After the work of J. Nuttall  \cite{Nut84}, it is well understood that in order to identify strong limits of orthogonal polynomials one needs to replace the potential-theoretic extremal problem with a boundary value problem on a certain Riemann surface. To this end, let $c\in(0,1/2]$ and $\Delta_{\sigma,c}$ be as in Proposition~\ref{prop:nikishin}. We define the Riemann surface corresponding to $c$, say $\RS_c$, through its realization in the following way. Take $3$ copies of $\overline\CP$. Cut one of them along the interval $\Delta_{\sigma,c}$, which henceforth is denoted by $\RS_c^{(0)}$, cut the second one, $\RS_c^{(1)}$, along $\Delta_\mu\cup\Delta_{\sigma,c}$, and the last one, $\RS_c^{(2)}$, along $\Delta_\mu$. To finish the construction, glue the banks of the corresponding cuts crosswise, see Figure~\ref{fig:surface}.

\begin{figure}[ht!]
\centering
\includegraphics[scale=.5]{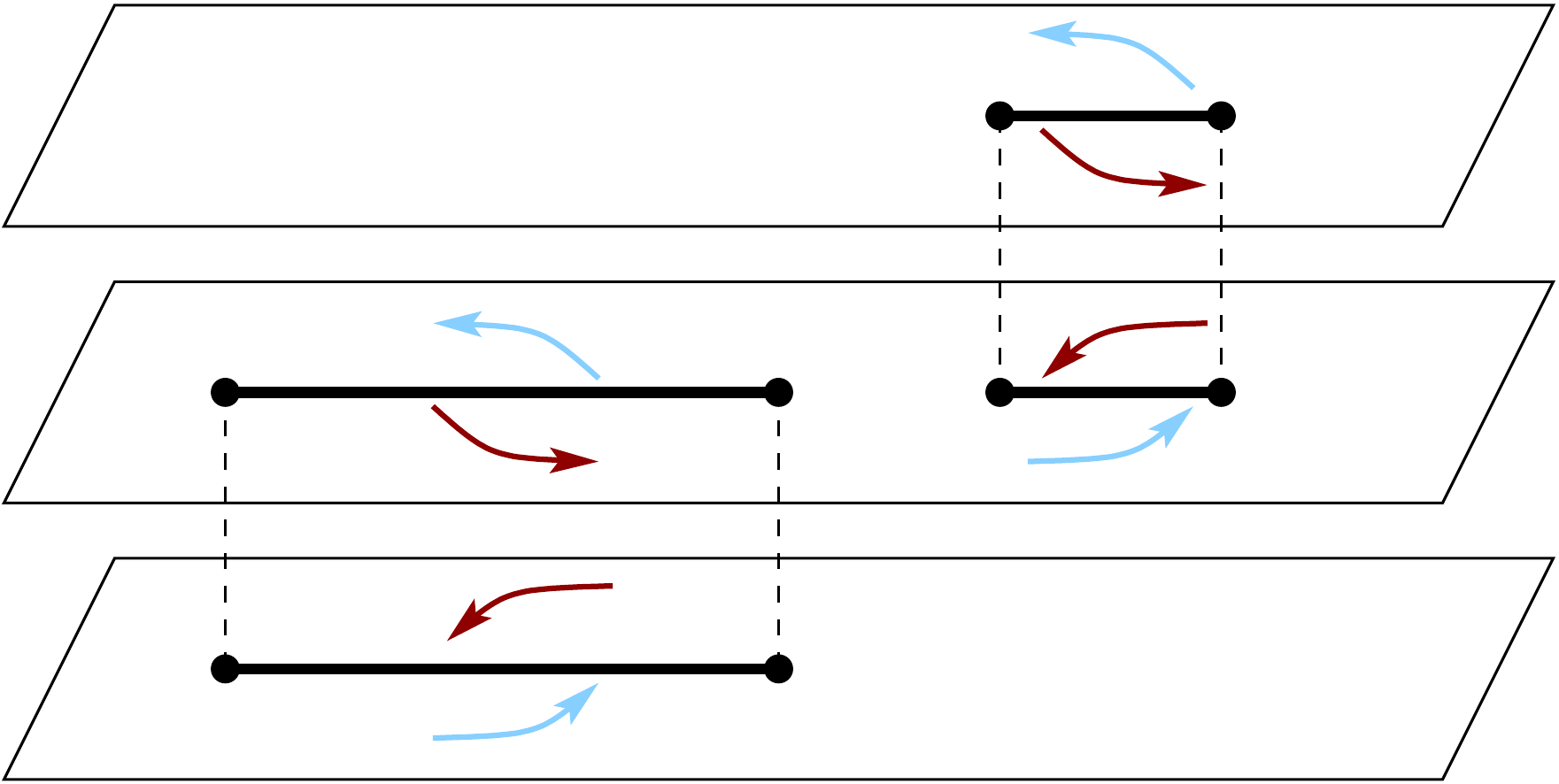}
\begin{picture}(0,0)
\put(-35,113){$\RS_c^{(0)}$}
\put(-35,68){$\RS_c^{(1)}$}
\put(-35,23){$\RS_c^{(2)}$}
\put(-235,13){$\boldsymbol b_\mu$}
\put(-127,13){$\boldsymbol a_\mu$}
\put(-235,58){$\boldsymbol b_\mu$}
\put(-127,58){$\boldsymbol a_\mu$}
\put(-108,58){$\boldsymbol a_\sigma$}
\put(-55,58){$\boldsymbol b_{\sigma,c}$}
\put(-108,105){$\boldsymbol a_\sigma$}
\put(-55,105){$\boldsymbol b_{\sigma,c}$}
\end{picture}
\caption{\small Riemann surface $\RS_c$ and its branch points $\boldsymbol a_\mu,\boldsymbol b_\mu,\boldsymbol a_{\sigma,c},\boldsymbol b_{\sigma,c}$.}
\label{fig:surface}
\end{figure}

We denote by $\pi$ the natural projection from $\RS_c$ to $\overline\CP$. We shall employ the notation $\z$ for a generic point of $\RS_c$ and use the convention $\pi(\z)=z$. If we want to specify the sheet of the surface, we write $z^{(i)}$ for a point on $\RS_c^{(i)}$ with $\pi(z^{(i)})=z$. This notation is well defined everywhere outside of the cycles $\boldsymbol\Delta_\mu:=\pi^{-1}(\Delta_\mu)$ and $\boldsymbol\Delta_{\sigma,c}:=\pi^{-1}(\Delta_{\sigma,c})$. Given a function $F(\z)$ defined on a subset of $\RS_c$, we set $F^{(i)}(z):=F\big(z^{(i)}\big)$ to be the pull-back from the $i$-th sheet.

Among all such surfaces, the ones with $c=\frac n{n+m}$ are especially important to us. We shall denote them by $\RS_{m,n}$. Observe that any $\RS_c$ has genus $0$. Thus, one can arbitrarily prescribe zero/pole multisets of rational functions on them as long as the multisets have the same cardinality. In what follows, we denote by $\Phi_{m,n}$ the rational function on $\RS_{m,n}$ with the divisor\footnote{The divisor is a formal expression that describes all the zeros (preceded by positive integer indicated multiplicity) and poles (preceded by negative integer also indicating multiplicity) of the function.} $(n+m)\infty^{(2)}-n\infty^{(0)}-m\infty^{(1)}$ and the normalization
\begin{equation}
\label{normalization}
\Phi_{m,n}^{(0)}(z)\Phi_{m,n}^{(1)}(z)\Phi_{m,n}^{(2)}(z) \equiv 1.
\end{equation}
Such a normalization is indeed possible since the function $\log\prod_{k=0}^2|\Phi_{m,n}^{(k)}|$ extends to a harmonic function on $\CP$ which has a well defined limit at infinity. Hence, it is a constant. Therefore, if \eqref{normalization} holds at one point, it holds throughout $\overline\CP$. It is a simple argument using Schwarz reflection principle, equilibrium relations \eqref{nikishin}, and the fact that only bounded harmonic function on $\RS_c$ are constants to show that
\begin{equation}
\label{logPhi}
\frac1{n+m}\log|\Phi_{m,n}(\z)| =
\left\{
\begin{array}{ll}
V^{-\tau_{\sigma,c}}(z) + \ell_{\mu,c} + 2\ell_{\sigma,c}, & \z\in\RS_{m,n}^{(0)}, \medskip \\
V^{\tau_{\sigma,c}-\tau_{\mu,c}}(z) + \ell_{\mu,c} - \ell_{\sigma,c}, & \z\in\RS_{m,n}^{(1)}, \medskip \\
V^{\tau_{\mu,c}}(z) -2\ell_{\mu,c} - \ell_{\sigma,c}, & \z\in\RS_{m,n}^{(2)},
\end{array}
\right.
\end{equation}
where $c=\frac n{n+m}$. Representation \eqref{logPhi} is not the only way to understand functions $\Phi_{m,n}$. Define
\begin{equation}
\label{hmn}
h_{m,n}(\z) :=
\left\{
\begin{array}{rl}
\displaystyle \int\frac{\mathrm d\tau_{\sigma,c}(x)}{z-x}, & \z\in\RS_{m,n}^{(0)}, \medskip \\
\displaystyle  \int\frac{\mathrm d(\tau_{\mu,c}-\tau_{\sigma,c})(x)}{z-x}, & \z\in\RS_{m,n}^{(1)}, \medskip \\
\displaystyle -\int\frac{\mathrm d\tau_{\mu,c}(x)}{z-x}, & \z\in\RS_{m,n}^{(2)},
\end{array}
\right.
\end{equation}
where, again, $c=\frac n{n+m}$. One can readily observe that
\[
h_{m,n}(z) = 2\partial_z\left(\frac1{n+m}\log|\Phi_{m,n}(\z)|\right)
\]
by \eqref{logPhi} and \eqref{hmn}, where $2\partial_z:=\partial_x-\mathrm i\partial_y$. As $\partial_z$-derivative of a harmonic function is holomorphic, $h_{m,n}$ is a rational function on $\RS_{m,n}$. It also follows from the above relation, that $h_{m,n}$ is the logarithmic derivative of $\Phi_{m,n}$ and therefore
\begin{equation}
\label{exph}
\Phi_{m,n}(\z) = \exp\left\{(n+m)\int^\z h_{m,n}(\x)\mathrm dx\right\},
\end{equation}
where the initial bound for integration is chosen so \eqref{normalization} holds. Moreover, we also can describe the divisor of $h_{m,n}$.

\begin{proposition}
\label{prop:h}
Given $n\leq m$, let $c=\frac n{n+m}$ and $h_{m,n}$ be defined by \eqref{hmn}.  Denote the endpoints of $\Delta_\nu$ by $a_\nu$ and $b_\nu$, $\nu\in\{\mu,\sigma\}$, and arrange them so that
\[
\text{either} \quad b_\mu < a_\mu < a_\sigma < b_\sigma \quad \text{or} \quad b_\sigma < a_\sigma < a_\mu < b_\mu.
\]
If, using the same convention, we denote the endpoints of $\Delta_{\sigma,c}$ by $a_{\sigma,c}$ and $b_{\sigma,c}$, then $a_{\sigma,c}=a_\sigma$. Moreover,  the divisor of $h_{m,n}$ is given by
\begin{equation}
\label{divisor}
\infty^{(0)} + \infty^{(1)} + \infty^{(2)} + \z_{m,n} - \boldsymbol a_\mu - \boldsymbol b_\mu - \boldsymbol a_\sigma - \boldsymbol b_{\sigma,c},
\end{equation}
where $\boldsymbol a_\mu$, $\boldsymbol b_\mu$, $\boldsymbol a_\sigma$, $\boldsymbol b_{\sigma,c}$ are the branch points of $\RS_{m,n}$ with the corresponding projections $a_\mu$, $b_\mu$, $a_\sigma$, $b_{\sigma,c}$, and $\z_{m,n}\in\RS_{m,n}^{(1)}$ with
\[
\pi(\z_{m,n})\in\left\{
\begin{array}{rll}
\big[b_{\sigma,c},\infty\big) & \text{if} & a_\sigma<b_\sigma, \medskip \\
\big(-\infty,b_{\sigma,c}\big] & \text{if} & b_\sigma<a_\sigma.
\end{array}
\right.
\]
Furthermore, $\z_{m,n}=\boldsymbol b_{\sigma,c}$ if and only if $b_{\sigma,c}\in\partial D_{\sigma,c}^-$, see \eqref{domains} and Figure~\ref{fig:domains}, that is, if and only if the domain $D_{\sigma,c}^-$ touches the interval $\Delta_{\sigma,c}$ (observe also that $b_{\sigma,c}=b_\sigma$ if $b_{\sigma,c}\not\in\partial D_{\sigma,c}^-$ since $\Delta_\sigma\setminus\Delta_{\sigma,c}\subset D_{\sigma,c}^-$ by Proposition~\ref{prop:nikishin}).
\end{proposition}

We prove Proposition~\ref{prop:h} in Section~\ref{ssec:Phi}. We clearly see from Proposition~\ref{prop:h} that the function $h_{m,n}$ is algebraic. More precisely, Proposition~\ref{prop:h} yields the following.

\begin{corollary}
If $b_{\sigma,c}\not\in\partial D_{\sigma,c}^-$, in which case $b_{\sigma,c}=b_\sigma$, then $h_{m,n}$ is the solution of the algebraic equation
\begin{equation}
\label{eqh1}
h^3-(1-\varkappa)\frac{P_2(z)}{\Pi(z)}\,h - \varkappa \frac{P_1(z)}{\Pi(z)}=0, \quad \varkappa=c-c^2,
\end{equation}
where  $c=\frac n{n+m}$,
\[
\Pi(z)=(z-a_{\mu})(z-b_\mu)(z-a_{\sigma}) (z-b_\sigma),
\]
and the polynomials $P_j$ are monic and of degree $j$, $j=1,2$. The three zeros of the polynomials $P_1$ and $P_2$ are determined by the three conditions that the discriminant of \eqref{eqh1}, i.e.,
\[
\frac1{\Pi^3(z)}\left[\left(\frac{1-\kappa}3P_2(z)\right)^3-\left(\frac\kappa2P_1(z)\right)^2\Pi(z)\right],
\]
has zeros of even multiplicity only and that the Riemann surface of the solution of \eqref{eqh1} must be as on Figure~\ref{fig:surface}.

If $b_{\sigma,c}\in\partial D_{\sigma,c}^-$, in which case $\z_{m,n}=\boldsymbol b_{\sigma,c}$, then $h_{m,n}$ is the solution of the algebraic equation
\begin{equation}
\label{eqh2}
h^3-(1-\varkappa)\frac{\widetilde{P}_{1}(z)}{\Pi(z)}\,h-\frac{\varkappa}{\Pi(z)}=0,
\end{equation}
where this time $\Pi(z)=(z-a_{\mu})(z-b_\mu)(z-a_{\sigma})$ and the only zero of the monic polynomial $\widetilde{P}_1$ is determined analogously to the first case.
\end{corollary}

Let us point out that if we take $a_{\mu}=a_{\sigma}=:a$ in \eqref{eqh2}, then $\widetilde{P}_1(z)=z-a$ and \eqref{eqh2} becomes
\begin{equation}
\label{eqh3}
h^3-\frac{(1-\varkappa)}{(z-b_{\mu})(z-a)}\,h - \frac{\varkappa}{(z-b_{\mu})(z-a)^{2}}=0.
\end{equation}
The only zero of the discriminant of \eqref{eqh3} is exactly $b_{\sigma,c}$ and is equal to
\begin{equation}\label{eqh4}
b_{\sigma,c}=\displaystyle\frac{\left(\frac{1-\varkappa}{3}\right)^{3}a-
\left(\frac{\varkappa}{2}\right)^{2}b_\mu}
{\left(\frac{1-\varkappa}{3}\right)^{3}-\left(\frac{\varkappa}{2}\right)^{2}}.
\end{equation}
Explicit expression \eqref{eqh3} allows us to numerically compute the boundary $\partial D_{\sigma,c}^+$, which is the trajectory $\Re\left[ {(h_{m,n}^{(0)}(z)-h_{m,n}^{(1)}(z))\mathrm dz}\right]=0$ emanating from $b_{\sigma,c}$, see Figure~\ref{fig:DivZone}.
\begin{figure}[ht!]
\centering
\includegraphics[scale=.4]{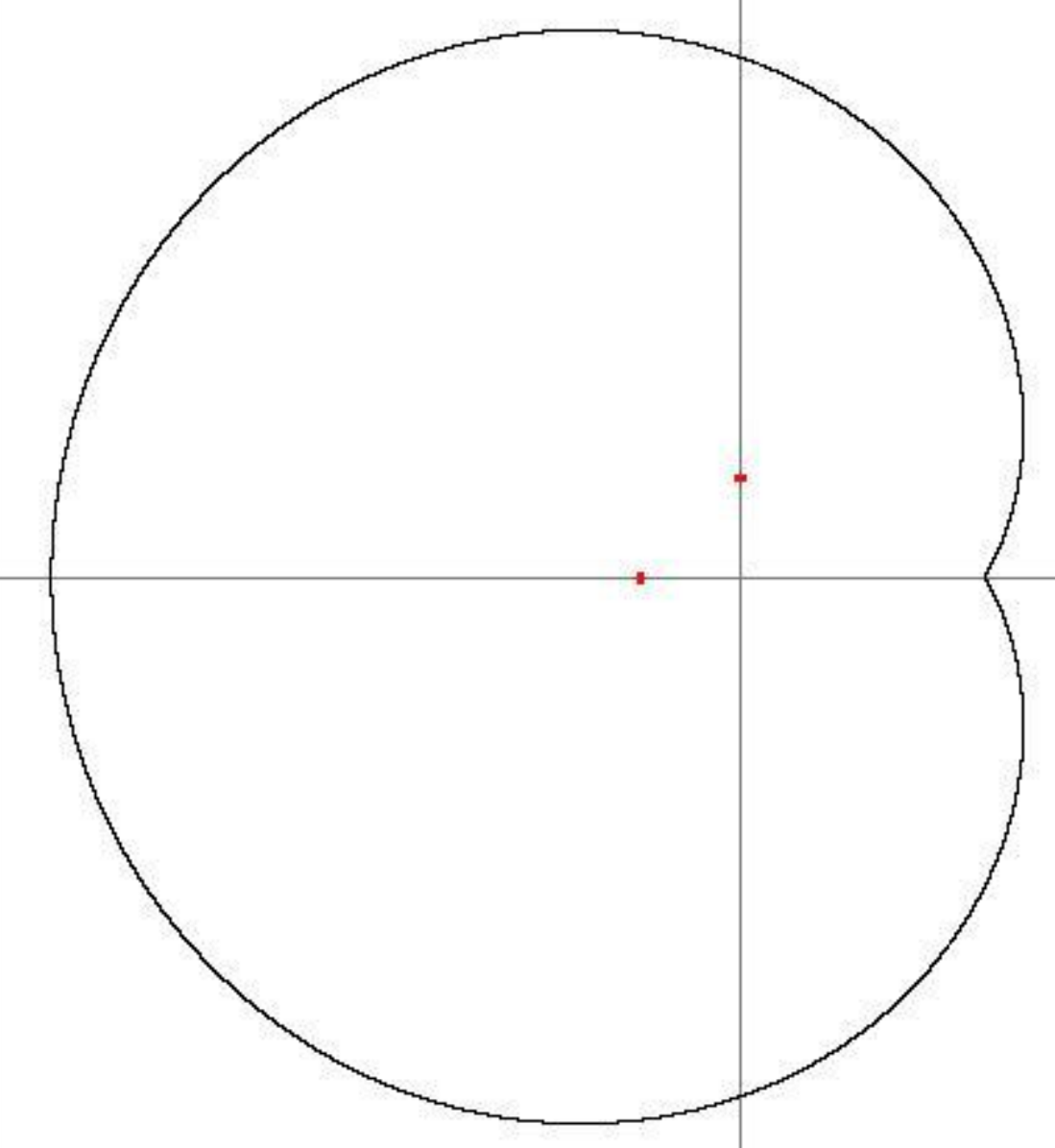}
\begin{picture}(0,0)
\put(-90,119){$\partial D_{\sigma,c}^+$}
\put(-41,69){\circle*{1}}
\put(-71,68){$\phantom{a}_{-1}$}
\put(-42,68){$\phantom{a}_{0}$}
\put(-36,80){$i$}
\put(-7,65){$b_{\sigma,c}$}
\put(-11,68){\circle*{1}}
\end{picture}
\caption{\small The curve $\partial D_{\sigma,c}^+$ numerically computed for parameters $b_\mu=-1$, $a_\mu=a_\sigma=a=0$, and $c=1/3$ (in this case $b_{\sigma,c}=2.43$).}
\label{fig:DivZone}
\end{figure}

Let us now specify which measures $\mu$ and $\sigma$ we consider. We shall assume that
\begin{equation}
\label{measures}
\mathrm d\nu(x) = \frac{\rho_\nu(x)}{2\pi\mathrm i}\frac{\mathrm dx}{w_\nu^+(x)}, \quad \nu\in\{\mu,\sigma\},
\end{equation}
where $\rho_\nu$ is a non-vanishing and holomorphic function in some neighborhood of $\Delta_\nu$ and
\[
w_\nu(z) := \sqrt{(z-a_\nu)(z-b_\nu)}
\]
is the branch holomorphic in $\mathbb C\setminus\Delta_\mu$ and normalized so that $w_\nu(z)/z\to1$ as $z\to\infty$. We define $w_{\sigma,c}(z)$ analogously.

As expected from the classical theory of orthogonal polynomials, we need to introduce an appropriate analog of the Szeg\H{o} function for the measures $\mu$ and $\sigma$. This is precisely the content of Proposition~\ref{prop:Szego} below. Its statement is a direct application of  \cite[Proposition~4]{uYa} with $\rho_1=w_\sigma^+/(\rho_\sigma (w_{\sigma,c}^+)^2)$ and $\rho_2=\rho_\mu/(w_{\sigma,c}w_\mu^+)$ (one needs to notice that the labeling of the sheets $\RS^{(0)}$ and $\RS^{(1)}$ is reversed there and the restriction $\alpha_{ij}>-1$ in \cite[Eq.~(23)]{uYa} is needed to make functions $\rho_i$ integrable and is not important for \cite[Proposition~4]{uYa} itself).

\begin{proposition}
\label{prop:Szego}
There exists a holomorphic and non-vanishing function on $\RS_c\setminus(\boldsymbol\Delta_\mu\cup\boldsymbol\Delta_{\sigma,c})$, say $S_c$, that has continuous traces on $\boldsymbol\Delta_\mu\cup\boldsymbol\Delta_{\sigma,c}\setminus\{\boldsymbol a_\mu,\boldsymbol b_\mu,\boldsymbol a_\sigma,\boldsymbol b_{\sigma,c}\}$, satisfies
\begin{equation}
\label{S-jump}
S_c^{(1)\pm}(x) =
\left\{
\begin{array}{ll}
S_c^{(0)\mp}(x)(\rho_\sigma w_{\sigma,c}^+/w_\sigma^+)(x), & x\in\Delta_{\sigma,c}^\circ, \medskip \\
S_c^{(2)\mp}(x)(w_{\sigma,c}/\rho_\mu)(x), & x\in\Delta_\mu^\circ,
\end{array}
\right.
\end{equation}
where $\Delta^\circ$ is the interior of the closed interval $\Delta$, is bounded around $\boldsymbol a_\mu,\boldsymbol b_\mu,\boldsymbol a_\sigma$ as well as $\boldsymbol b_{\sigma,c}$ when $b_{\sigma,c}=b_\sigma$, and behaves like $\big|S_c^{(1)}(z)\big| \sim \big|S_c^{(0)}(z)\big|^{-1} \sim |z-b_{\sigma,c}|^{1/4}$ as $z\to b_{\sigma,c}\neq b_\sigma$. Moreover, it holds that $S_c^{(0)}S_c^{(1)}S_c^{(2)}\equiv1$.
\end{proposition}

Now we are ready to state our main result.

\begin{theorem}
\label{thm:Frobenius}
Let $\mu$ and $\sigma$ be of the form \eqref{measures} and assume that $\mu$ possesses the full system of orthonormal polynomials. Assume further that  $n-1\leq m$ and $\frac n{n+m}\to c>0$ as $n\to\infty$. Then for all $n$ large, $(m,n)$-th Frobenius-Pad\'e approximant $P_{m,n}/Q_{m,n}$ is unique and $\deg(Q_{m,n})=n$. Moreover, if $K\subset \overline\CP\setminus\Delta_\sigma$ is closed, then
\begin{equation}
\label{asymptotics1}
\left\{
\begin{array}{ll}
Q_{m,n} = \big[ 1 + o(1) \big] \Phi_{m+1,n}^{(0)}S_c^{(0)} & \text{on} \quad K, \medskip \\
Q_{m,n} = \big[ 1 + o(1) \big] \Phi_{m+1,n}^{(0)+}S_c^{(0)+} + \big[ 1 + o(1) \big] \Phi_{m+1,n}^{(0)-}S_c^{(0)-} & \text{on} \quad \Delta_{\sigma,c}^\circ,
\end{array}
\right.
\end{equation}
where $o(1)$ is uniform on $K$ and locally uniform on $\Delta_{\sigma,c}^\circ$; however, if $\Delta_{\sigma,c}\cap\partial D_{\sigma,c}^-=\varnothing$, then $\Delta_{\sigma,c}=\Delta_\sigma$ and $o(1)=\mathcal O(C_{\mu,\sigma}^{-n})$ for some constant $C_{\mu,\sigma}>1$ with the second equality holding uniformly on $\Delta_\sigma$. Furthermore, if $K\subset \overline\CP\setminus(\Delta_\sigma\cup\Delta_\mu)$ is closed, then
\begin{equation}
\label{asymptotics2}
\left\{
\begin{array}{ll}
w_{\sigma,c} R_{m,n} =   \big[ 1 + o(1) \big] \Phi_{m+1,n}^{(1)}S_c^{(1)}, & \text{on} \quad K, \medskip \\
w_{\sigma,c}^\pm R_{m,n}^\pm =   \big[ 1 + o(1) \big] \Phi_{m+1,n}^{(1)\pm}S_c^{(1)\pm}, & \text{on} \quad \Delta_{\sigma,c}^\circ, \medskip \\
w_{\sigma,c} R_{m,n}  =  \big[ 1 + o(1) \big] \Phi_{m+1,n}^{(1)+}S_c^{(1)+} +  \big[ 1 + o(1) \big] \Phi_{m+1,n}^{(1)-}S_c^{(1)-}, & \text{on} \quad \Delta_\mu^\circ,
\end{array}
\right.
\end{equation}
where $R_{m,n}=Q_{m,n}\widehat\sigma-P_{m,n}$ and $o(1)$ has the same properties as in \eqref{asymptotics1}.
\end{theorem}

\begin{remark}
Polynomial $Q_{m,n}$ is defined up to a multiplicative constant. However, choosing $Q_{m,n}$ uniquely determines $P_{m,n}$, and respectively $R_{m,n}$. Polynomials $Q_{m,n}$ in \eqref{asymptotics1} are normalized so that the leading coefficient is equal to the coefficient of $\Phi_{m+1,n}^{(0)}S_c^{(0)}$ next to $z^n$ when the latter function is developed into a power series at infinity.
\end{remark}

\begin{remark}
The proof of Theorem~\ref{thm:Frobenius} follows the framework of Riemann-Hilbert analysis for orthogonal polynomials formulated by Fokas, Its, and Kitaev \cite{FIK91,FIK92}, in which Frobenius-Pad\'e approximants are characterized via a certain Riemann-Hilbert problem whose solution is obtained using a variation of Deift and Zhou steepest descent method \cite{DZ93}. In this realm of ideas it is well understood that one can introduce Fisher-Hartwig singularities into \eqref{measures}. That is, \eqref{measures} can be replaced by
\[
\mathrm d\nu(x) = \rho_\nu(x)\prod_{i=0}^{I_\nu}|x-x_{i,\nu}|^{\alpha_{i,\nu}}\prod_{i=1}^{I_\nu}\left\{ \begin{array}{ll} 1, & x<x_{i,\nu} \\ \beta_{i,\nu}, &x>x_{i,\nu} \end{array} \right\}\mathrm dx,
\]
where $\rho_\nu$ is as before, $a_\nu=x_{0,\nu}<x_{1,\nu}<\cdots<x_{I_\nu-1,\nu}<x_{I_\nu,\nu}=b_\nu$, $\alpha_{i,\nu}>-1$, and $\beta_{i,\nu}\not\in(-\infty,0]$, \cite{Van03,FMMFSou10,FMMFSou11,DIKr11,uYa}. Implementing such a modification is rather lengthy as details are very technical and does not provide any additional insight on the behavior of the approximants. Thus, we opted to consider only the measures of the form \eqref{measures}.
\end{remark}

\begin{remark}
As was noticed in \cite{Bogolubs16}, in the case of positive measures $\mu$ and $\sigma$ in \eqref{expansion}, \eqref{markov}, the statement of Theorem~\ref{thm:Frobenius} for the diagonal sequence  $(n-1,n)$ follows from the theorem on strong asymptotics of multiple orthogonal polynomials from~\cite{Apt99}.
\end{remark}

\begin{remark}
By the definition of the linear forms $R_{m,n}$, \eqref{asymptotics1}--\eqref{asymptotics2}, and \eqref{logPhi}, it holds that the error of approximation by Frobenius-Pad\'e approximants behaves like
\[
\big|\widehat\sigma - P_{m,n}/Q_{m,n}\big| \sim \left| \Phi_{m+1,n}^{(1)}/\Phi_{m+1,n}^{(0)} \right|  = \exp\left\{-(n+m+1)\left(V^{\tau_{\mu,c}-2\tau_{\sigma,c}}+3\ell_{\sigma,c}\right)\right\}.
\]
Hence, the approximants converge to $\widehat\sigma$ uniformly on compact subsets of $D_{\sigma,c}^+$ and diverge uniformly on compact subsets of $D_{\sigma,c}^-$. It also follows from Proposition~\ref{prop:nikishin}, that they converge locally uniformly in $\CP\setminus\Delta_\sigma$ only if $m=n+o(n)$.
\end{remark}

\section{Proofs}

\subsection{Functions $h_{m,n}$ and $\Phi_{m,n}$}
\label{ssec:Phi}

For the proof of Proposition~\ref{prop:h}, put $c=\frac n{n+m}$. To find the divisor of $h_{m,n}$, observe that $h_{m,n}$ is holomorphic everywhere outside of the four branch points of $\RS_{m,n}$ and at each of these points it can have at most a simple pole since $\Phi_{m,n}$ is bounded there. Clearly, $h_{m,n}$ has three simple zeros, one at each $\infty^{(k)}$, $k\in\{0,1,2\}$. If $\boldsymbol b_{\sigma,c}$ is not a pole, then the remaining three brach points of $\RS_{m,n}$ must be poles (the number of poles must be equal to the number of zeros) and there cannot be any more poles and/or zeros. In this case we put $\z_{m,n}=\boldsymbol  b_{\sigma,c}$, which verifies \eqref{divisor}. If $\boldsymbol b_{\sigma,c}$ is a pole of $h_{m,n}$, then
\begin{equation}
\label{htoinfty}
h_{m,n}^{(1)}(x) \to -(-1)^\iota\infty \quad \text{as} \quad x\to b_{\sigma,c}, \; \R\ni x\not\in\Delta_{\sigma,c}, \quad \iota:=\left\{
\begin{array}{ll}
0, & a_\sigma < b_\sigma, \\
1, & b_\sigma < a_\sigma,
\end{array}
\right.
\end{equation}
by the very definition of $h_{m,n}^{(1)}$ in \eqref{hmn} including the positivity of $\tau_{\sigma,c}$ and our labeling convention for the endpoints of $\Delta_{\sigma,c}$. On the other hand, it holds that
\[
h_{m,n}^{(1)}(z)=\frac{|\tau_{\mu,c}|-|\tau_{\sigma,c}|} z+\mathcal O\left(z^{-2}\right) \quad \text{as} \quad z\to\infty,
\]
again, by the very definition of $h_{m,n}$. As $|\tau_{\mu,c}|-|\tau_{\sigma,c}|=1-c>0$, we have that
\begin{equation}
\label{hpositive}
(-1)^\iota h_{m,n}^{(1)}(x) >0 \quad \text{as} \quad \R\ni x\to (-1)^\iota\infty.
\end{equation}
Therefore there indeed exists $z_{m,n}$ between $b_{\sigma,c}$ and $(-1)^\iota\infty$ such that $h_{m,n}^{(1)}(z_{m,n})=0$. Since $h_{m,n}$ has three more zeros, the rest of the branch points must be poles as claimed.

Assume that $\boldsymbol b_{\sigma,c}$ is a pole of $h_{m,n}$ (equivalently $\boldsymbol b_{\sigma,c}\neq\z_{m,n}$). Since $\RS_{m,n}$ has square root branching at $\boldsymbol b_{\sigma,c}$, it follows from \eqref{htoinfty} and the fact that the sum $h_{m,n}^{(0)}+h_{m,n}^{(1)}$ is holomorphic around $b_{\sigma,c}$ that
\[
\left\{
\begin{array}{l}
h_{m,n}^{(1)}(z) = d_{m,n}\big((-1)^\iota(z-b_{\sigma,c})\big)^{-1/2} + \mathcal O(1), \medskip \\
h_{m,n}^{(0)}(z) = -d_{m,n}\big((-1)^\iota(z-b_{\sigma,c})\big)^{-1/2} + \mathcal O(1),
\end{array}
\right. \quad \text{as} \quad z\to b_{\sigma,c}, \; z\not\in\Delta_{\sigma,c},
\]
where $(-1)^\iota d_{m,n}<0$, and the square root is principal. It further follows from the above asymptotics as well as from \eqref{logPhi} and \eqref{exph} that
\[
V^{\tau_{\mu,c}-2\tau_{\sigma,c}}(z) + 3\ell_{\sigma,c} = \re\left(\int_{b_{\sigma,c}}^z\left(h_{m,n}^{(0)} - h_{m,n}^{(1)}\right)(y)\mathrm dy\right)>0
\]
as $z\to b_{\sigma,c}$, $z\not\in\Delta_{\sigma,c}$. Hence, $b_{\sigma,c}\not\in\partial D_{\sigma,c}^-$. On the other hand, if $\boldsymbol b_{\sigma,c}$ is not a pole of $h_{m,n}$ (equivalently $\boldsymbol b_{\sigma,c}=\z_{m,n}$), then
\[
\left\{
\begin{array}{l}
h_{m,n}^{(0)}(z) = h_{m,n}(\boldsymbol b_{\sigma,c}) + e_{m,n}\big((-1)^\iota(z-b_{\sigma,c})\big)^{1/2} + \mathcal O\big(|z-b_{\sigma,c}|\big), \medskip \\
h_{m,n}^{(1)}(z) = h_{m,n}(\boldsymbol b_{\sigma,c}) - e_{m,n}\big((-1)^\iota(z-b_{\sigma,c})\big)^{1/2} + \mathcal O\big(|z-b_{\sigma,c}|\big),
\end{array}
\right.
\]
as $z\to b_{\sigma,c}$, $z\not\in\Delta_{\sigma,c}$, since $h_{m,n}^{(0)}+h_{m,n}^{(1)}$ is holomorphic around $b_{\sigma,c}$. Moreover, as $h_{m,n}^{(0)}$ satisfies \eqref{hpositive} and is monotone between $b_{\sigma,c}$ and $(-1)^\iota\infty$, it holds that $(-1)^\iota e_{m,n}<0$. Therefore,
\[
V^{\tau_{\mu,c}-2\tau_{\sigma,c}}(x) + 3\ell_{\sigma,c} = \int_{b_{\sigma,c}}^x\left(h_{m,n}^{(0)} - h_{m,n}^{(1)}\right)(y)\mathrm dy<0
\]
for $x\to b_{\sigma,c}$, $\R\ni x\not\in\Delta_{\sigma,c}$. In particular, $b_{\sigma,c}\in\partial D_{\sigma,c}^-$. This finishes the proof of the last claim of the proposition. Finally, similar analysis can be used to show that $a_{\sigma,c}\not\in\partial D_{\sigma,c}^-$, which, as noted at the end of Proposition~\ref{prop:h}, implies the equality $a_{\sigma,c}=a_\sigma$.

For the future use let us record several facts. Firstly, it holds that
\begin{equation}
\label{estPhi21}
\left|\Phi_{m,n}^{(2)}/\Phi_{m,n}^{(1)}\right| < 1 \quad \text{in} \quad \overline\CP\setminus\Delta_\mu.
\end{equation}
Indeed, \eqref{estPhi21} is equivalent to $V^{\tau_{\sigma,c}}-2V^{\tau_{\mu,c}} + 3\ell_{\mu,\sigma} >0$ by \eqref{logPhi}. The left-hand side of this inequality is superharmonic in $\CP\setminus\Delta_\mu$, is identically zero on $\Delta_\mu$ by \eqref{nikishin}, and approaches $+\infty$ as $z\to\infty$ since $|\tau_{\sigma,c}|=c<2=2|\tau_{\mu,c}|$. The desired inequality now follows from the minimum principle for superharmonic functions \cite[Theorem~2.3.1]{Ransford}.

Secondly, let $\{c_n\}$ be a sequence such that $c_n\to c>0$ as $n\to\infty$, $c_n\leq 1/2$. Then
\begin{equation}
\label{cws}
\tau_{\mu,c_n}\cws \tau_{\mu,c} \quad \text{and} \quad \tau_{\sigma,c_n} \cws \tau_{\sigma,c},
\end{equation}
where $\cws$ stands for the weak$^*$ convergence of measures. Indeed, besides \eqref{nikishin}, the pair $(\tau_{\mu,c},\tau_{\sigma,c})$ is characterized as the unique minimizers in $\mathcal M_c$ of the energy functional
\[
J(\tau_\mu,\tau_\sigma) := I(\tau_\mu,\tau_\mu) +  I(\tau_\sigma,\tau_\sigma) -  I(\tau_\mu,\tau_\sigma),
\]
where $I(\nu,\lambda):=-\int\log|z-w|\mathrm d\nu(z)\mathrm d\lambda(w)$, see \cite{Nik86,GRakhS97}. Let $\tau_\mu$ and $\tau_\sigma$ be weak$^*$ limit points of $\{\tau_{\mu,c_n}\}$ and $\{\tau_{\sigma,c_n}\}$, respectively. Clearly, $(\tau_\mu,\tau_\sigma)\in \mathcal M_c$. Then
\[
J(\tau_{\mu,c},\tau_{\sigma,c}) = \lim_{n\to\infty}J\left(\tau_{\mu,c},\frac{c_n}c\tau_{\sigma,c}\right) \geq \liminf_{n\to\infty} J(\tau_{\mu,c_n},\tau_{\sigma,c_n}) \geq J(\tau_\mu,\tau_\sigma),
\]
where the first inequality follows from the fact that $(\tau_{\mu,c_n},\tau_{\sigma,c_n})$ is the minimizer of the $J$-functional in $\mathcal M_{c_n}$ and the second inequality is the consequence of the principle of descent \cite[Theorem~I.6.8]{SaffTotik}, i.e, $\liminf I(\tau_{\nu,c_n},\tau_{\nu,c_n})\geq I(\tau_\nu,\tau_\nu)$, and the fact that $\Delta_\mu\cap\Delta_\sigma=\varnothing$ (in this case the kernel $\log|z-w|$ is continuous on $\Delta_\mu\times\Delta_\sigma$ and therefore $I(\tau_{\mu,c_n},\tau_{\sigma,c_n})\to I(\tau_\mu,\tau_\sigma)$ by weak$^*$ convergence of measures). As $(\tau_{\mu,c},\tau_{\sigma,c})$ is the unique minimizer of the $J$-functional in $\mathcal M_c$, \eqref{cws} follows.

Finally, let us point out that in the above setting $b_{\sigma,c_n}\to b_{\sigma,c}$, and
\begin{equation}
\label{convVs}
V^{\tau_{\nu,c_n}} \to V^{\tau_{\nu,c}} \quad \text{locally uniformly in} \quad \overline\CP\setminus\Delta_{\nu,c}, \quad \nu\in\{\mu,\sigma\},
\end{equation}
as $n\to\infty$, which is an immediate consequence of \eqref{cws}.

\subsection{Riemann-Hilbert Problem for Frobenius-Pad\'e Approximants}

Given such a pair of integers $(m,n)$, $n-1\leq m$, we are interested in finding a $3\times3$ matrix-valued function $\boldsymbol Y$ that solves the following Riemann-Hilbert Problem (\rhy):
\begin{itemize}
\label{rhy}
\item[(a)] ${\boldsymbol Y}$ is analytic in $\CP\setminus(\Delta_\mu\cup\Delta_\sigma)$ and
\[
\lim_{z\to\infty} {\boldsymbol Y}(z)\ \diag\left(z^{-n},z^{-m-1},z^{n+m+1}\right) = \boldsymbol I,
\]
where $\diag(\cdot,\cdot,\cdot)$ is the diagonal matrix and ${\boldsymbol I}$ is the identity matrix;
\item[(b)] ${\boldsymbol Y}$ has continuous traces on $\Delta_\mu^\circ\cup\Delta_\sigma^\circ$ that satisfy
\[
\boldsymbol Y_+ = \boldsymbol Y_- \mathsf T_\nu \left(\begin{matrix} 1 & \rho_\nu/w_\nu \\ 0 & 1 \end{matrix}\right) \quad \text{on} \quad \Delta_\nu^\circ, \quad \nu\in\{\mu,\sigma\},
\]
where transformations $\mathsf T_\mu$ and $\mathsf T_\sigma$ act on $2\times2$ matrices in the following fashion:
\[
\mathsf T_\mu\boldsymbol A := \left(\begin{matrix} 1 & 0 & 0 \\ 0 & [\boldsymbol A]_{11} & [\boldsymbol A]_{12} \\ 0 &  [\boldsymbol A]_{21} & [\boldsymbol A]_{22} \end{matrix}\right) \quad \text{and} \quad \mathsf T_\sigma\boldsymbol A := \left(\begin{matrix}   [\boldsymbol A]_{11} & [\boldsymbol A]_{12} & 0 \\ [\boldsymbol A]_{21} & [\boldsymbol A]_{22} & 0 \\ 0 & 0 & 1 \end{matrix}\right);
\]
\item[(c)] the entries of $\boldsymbol Y$ are bounded except for the second column around the endpoints of $\Delta_\sigma$ and the third column around the endpoint of $\Delta_\mu$ where they behave as $\boldsymbol{\mathcal O}(|z-e|^{-1/2})$ with $e$ being the corresponding endpoint.
\end{itemize}

To see how \hyperref[rhy]{\rhy} is connected to Frobenius-Pad\'e approximants, observe that the linear form $R_{m,n}$ is a holomorphic function in $\CP\setminus\Delta_\sigma$ with a pole of degree at most $m$ at infinity. Moreover, it follows from Plemelj-Sokhotski formulae \cite[Section~I.4.2]{Gakhov} and \eqref{measures} that
\begin{equation}
\label{R1-jump}
R_{m,n}^+ - R_{m,n}^- = Q_{m,n}\rho_\sigma/w_\sigma^+ \quad \text{on} \quad \Delta_\sigma^\circ.
\end{equation}
It is also known from the theory of boundary behavior of Cauchy integrals \cite[Section~I.8]{Gakhov} that $R_{m,n}(z)\sim |z-e|^{-1/2}$ as $z\to e\in\{a_\sigma,b_\sigma\}$. As mentioned before, condition \eqref{Frobenius} implies that $R_{m,n}$ is orthogonal to all polynomials of degree at most $m+n$ with respect to $\mu$, i.e.,
\begin{equation}
\label{Fs-ortho-1}
\int x^i R_{m,n}(x)\mathrm d\mu(x) =0, \quad i\in\{0,\ldots,m+n\}.
\end{equation}
Orthogonality relations \eqref{Fs-ortho-1} imply that the Cauchy transform of $R_{m,n}$ vanishes at infinity with order at least $m+n+2$.  That is, the function
\[
C_{m,n}(z) := \int\frac{R_{m,n}(x)}{x-z}\mathrm d\mu(x), \quad z\in\overline\CP\setminus\Delta_\mu,
\]
is a holomorphic function in $\overline\CP\setminus\Delta_\mu$, has a zero of order at least $m+n+2$ at infinity, and satisfies
\begin{equation}
\label{R2-jump}
C_{m,n}^+ - C_{m,n}^- = R_{m,n}\rho_\mu/w_\mu^+ \quad \text{on} \quad \Delta_\mu^\circ.
\end{equation}
As in the case of $R_{m,n}$, we can conclude that $C_{m,n}(z)\sim |z-e|^{-1/2}$ as $z\to e\in\{a_\mu,b_\mu\}$.

\begin{lemma}
\label{lem:uniqueness}
Let $n-1\leq m$. If $(m,n)$-th Frobenius-Pad\'e approximant is unique and $\deg(Q_{m,n})=n$, then $R_{m+1,n-1}(z)\sim z^{m+1}$ and $C_{m,n-1}(z)\sim z^{-(n+m+1)}$ as $z\to\infty$ for any $(m+1,n-1)$-st and $(m,n-1)$-st approximants, respectively.
\end{lemma}
\begin{proof}
Assume to the contrary that there is $(m,n-1)$-st approximant such that $C_{m,n-1}(z)\sim z^{-(n+m+j+1)}$ as $z\to\infty$ for some $j>0$. It can be readily verified that in this case the corresponding linear form $R_{m,n-1}$ is orthogonal to all polynomials of degree at most $m+n+j$.  Then we can conclude from \eqref{Frobenius} that this $(m,n-1)$-st approximant is also $(m+j_1,n-1+j_2)$-th approximant for any choice of $j_1,j_2\geq0$, $j_1+j_2\leq j$. By taking $j_1=1$ and $j_2=0$, we see that there exists $(m+1,n-1)$-st approximant for which $R_{m+1,n-1}(z)\sim z^{m+1-i}$ for some $i>0$ (recall that $n-1\leq m$). This implies that $\deg(P_{m+1,n-1}) = m+1-i\leq m$, and respectively, this Frobenius-Pad\'e approximant also corresponds to the index $(m,n)$ and its denominator has degree at most $n-1$.
\end{proof}

Assuming $(m,n)$-th approximant is the unique and $\deg(Q_{m,n})=n$, define
\begin{equation}
\label{Y}
\boldsymbol Y_{m,n} := \boldsymbol C_{m,n}\left(\begin{matrix}
Q_{m,n} & R_{m,n} & C_{m,n} \medskip \\
Q_{m+1,n-1} & R_{m+1,n-1} & C_{m+1,n-1} \medskip \\
Q_{m,n-1} & R_{m,n-1} & C_{m,n-1}
\end{matrix}\right),
\end{equation}
where $\boldsymbol C_{m,n}$ is a diagonal matrix of constants chosen so that $\boldsymbol Y_{m,n}$ satisfies the normalization at infinity from \hyperref[rhy]{\rhy}(a). The choice of $\boldsymbol C_{m,n}$ is always possible due to Lemma~\ref{lem:uniqueness}. Then the following lemma holds.

\begin{lemma}
\label{lem:rhy}
Let $n-1\leq m$. If \hyperref[rhy]{\rhy} is solvable, then $(m,n)$-th, $(m,n-1)$-st, and $(m+1,n-1)$-st Frobenius-Pad\'e approximants are unique, $\deg(Q_{m,n})=n$, and $\boldsymbol Y=\boldsymbol Y_{m,n}$.
\end{lemma}
\begin{proof}
Assume that \hyperref[rhy]{\rhy} is solvable and $\boldsymbol Y$ is a solution. We consider only the first row as the other ones can be analyzed similarly. It follows from \hyperref[rhy]{\rhy}(a,b) that $[\boldsymbol Y]_{11}$ must be polynomial of degree $n$. Further,  all three properties \hyperref[rhy]{\rhy}(a,b,c) imply that $[\boldsymbol Y]_{12} = [\boldsymbol Y]_{11}\widehat\sigma - P$ for some polynomial $P$, $\deg(P)\leq m$. Analogously, we see that $[\boldsymbol Y]_{13}$ must be a  Cauchy transform of $[\boldsymbol Y]_{12}\rho_\mu/w_\mu^+$. The vanishing of $[\boldsymbol Y]_{13}$ at infinity with order at least $m+n+2$ implies that $[\boldsymbol Y]_{12}$ is orthogonal to $x^i$, $i\in\{0,\dots,m+n\}$, with respect to $\mu$. Therefore, $c_i([\boldsymbol Y]_{11}\widehat\sigma - P)=0$ for such $i$, and, by definition, $P/[\boldsymbol Y]_{11}$ is an $(m,n)$-th Frobenius-Pad\'e approximant.

To show uniqueness of the approximants, observe first that the solution of \hyperref[rhy]{\rhy} is unique. Indeed, Let $\boldsymbol Y_1$ and $\boldsymbol Y_2$ be solutions. As the determinant of the jump matrix in \hyperref[rhy]{\rhy}(b) is $1$ and $\det(\boldsymbol Y_1)$ can have at most square root singularities at the endpoints of $\Delta_\mu$ and $\Delta_\sigma$ by \hyperref[rhy]{\rhy}(c), $\det(\boldsymbol Y_1)$ is an entire function such that $\det(\boldsymbol Y_1)(\infty)=1$. Hence, $\det(\boldsymbol Y_1)\equiv1$ and therefore $\boldsymbol Y_1$ is invertible. Then $\boldsymbol Y_2\boldsymbol Y_1^{-1}$ is an entire matrix-valued function that is equal to $\boldsymbol I$ at infinity. Thus, $\boldsymbol Y_2=\boldsymbol Y_1$.

Second, observe that if $(m,n)$-th approximant is unique and $\deg(Q_{m,n})=n$, then $\boldsymbol Y_{m,n}$ solves \hyperref[rhy]{\rhy}. Indeed, the fact that $\boldsymbol Y_{m,n}$ satisfies \hyperref[rhy]{\rhy}(a,c) easily follows from the analyticity properties and the behavior at infinity of $Q_{m,n}$, $R_{m,n}$, and $C_{m,n}$, as well as from the choice of $\boldsymbol C_{m,n}$. \hyperref[rhy]{\rhy}(b) is an immediate consequence of \eqref{R1-jump} and \eqref{R2-jump}.

Now, let $\boldsymbol Y$ be the solution. Assume $Q_{m,n}$, $R_{m,n}$, and $C_{m,n}$ correspond to another $(m,n)$-th approximant. Without loss of generality we can assume that $\deg(Q_{m,n})=n$ (otherwise we should take $[\boldsymbol Y]_{11}-Q_{m,n}$ instead of $Q_{m,n}$). Construct matrix $\boldsymbol Y_1$ by  replacing the first row of $\boldsymbol Y$ with $\big(\begin{matrix} Q_{m,n} & R_{m,n} & C_{m,n} \end{matrix}\big)$. From the first paragraph we know that the second and third rows of $\boldsymbol Y_1$ correspond to $(m+1,n-1)$-st and $(m,n-1)$-st approximants and therefore we deduce from the third paragraph that $\boldsymbol Y_1$ is a solution of \hyperref[rhy]{\rhy}. By uniqueness, we get that $\boldsymbol Y_1=\boldsymbol Y$ and therefore $(m,n)$-th approximant is unique. Thus, we know from Lemma~\ref{lem:uniqueness} that any $(m+1,n-1)$-st and $(m,n-1)$-st must satisfy its conclusions. Hence, if they were not unique, we could replace the second and third rows of $\boldsymbol Y$ by the functions coming from other approximants and obtain a solution of \hyperref[rhy]{\rhy} different form $\boldsymbol Y$, which is impossible.
\end{proof}

\subsection{Non-Linear Steepest Descent Analysis in the Case $\Delta_{\sigma,c}\cap\partial D_{\sigma,c}^-=\varnothing$}
\label{ssec:no-local}

Recall that in the considered case $\Delta_{\sigma,c}=\Delta_\sigma$, see Proposition~\ref{prop:nikishin}. Moreover, it follows from \eqref{cws} and Proposition~\ref{prop:h} that $\Delta_{\sigma,\frac n{n+m}}=\Delta_\sigma$ for all $n$ large enough. In particular, we have that $\RS_{m,n}=\RS_c=:\RS$ for all such $n$ and we consider only these indices from now on.

Let $\Gamma_\mu$ and $\Gamma_\sigma$ be positively oriented Jordan curves lying exterior to each other and containing $\Delta_\mu$ and $\Delta_\sigma$ in the respective interiors. We denote by $\Omega_\nu$ the domain delimited by $\Gamma_\nu\cup\Delta_\nu$, $\nu\in\{\mu,\sigma\}$. We assume that $\rho_\nu$ extends holomorphically across $\Gamma_\nu$, $\nu\in\{\mu,\sigma\}$, and that $\Gamma_\sigma\subset D_{\sigma,c}^+$. Observe that in the considered case $\partial D_{\sigma,\frac n{n+m}}^+$ approaches $D_{\sigma,c}^+$ by \eqref{convVs} and therefore $\Gamma_\sigma\subset D_{\sigma,\frac n{n+m}}^+$ is uniformly separated from $\partial D_{\sigma,\frac n{n+m}}^+$. Define
\begin{equation}
\label{X}
\boldsymbol X = \boldsymbol Y \left\{
\begin{array}{ll}
\mathsf T_\nu\left( \begin{matrix} 1 & 0  \\ -w_\nu/\rho_\nu & 1\end{matrix} \right), & \text{in} \quad \Omega_\nu, \quad \nu\in\{\mu,\sigma\}, \medskip \\
\boldsymbol I, & \text{otherwise}.
\end{array}
\right.
\end{equation}
It is easy to verify that $\boldsymbol X$ solves the following Riemann-Hilbert problem (\rhx):
\begin{itemize}
\label{rhx}
\item[(a)] $\boldsymbol X$ is analytic in $\CP\setminus(\Delta_\mu\cup\Delta_\sigma\cup\Gamma_\mu\cup\Gamma_\sigma)$ and
\[
\lim_{z\to\infty} {\boldsymbol X}(z)\ \diag\left(z^{-n},z^{-m-1},z^{n+m+1}\right) = \boldsymbol I;
\]
\item[(b)] $\boldsymbol X$ has continuous traces on $\Delta_\mu^\circ\cup\Delta_\sigma^\circ\cup\Gamma_\mu\cup\Gamma_\sigma$ that satisfy
\[
\boldsymbol X_+ = \boldsymbol X_- \left\{
\begin{array}{rl}
\mathsf T_\nu \left(\begin{matrix} 0 & \rho_\nu/w_\nu^+ \\ -w_\nu^+/\rho_\nu & 0 \end{matrix}\right) & \text{on} \quad \Delta_\nu^\circ, \medskip \\
\mathsf T_\nu \left(\begin{matrix} 1 & 0 \\ w_\nu/\rho_\nu & 1 \end{matrix}\right) & \text{on} \quad \Gamma_\nu,
\end{array}
\right. \quad \nu\in\{\sigma,\mu\};
\]
\item[(c)] $\boldsymbol X$ satisfies \hyperref[rhy]{\rhy}(c).
\end{itemize}

Then the following lemma can be easily checked.

\begin{lemma}
\label{lem:rhs}
\hyperref[rhx]{\rhx} is solvable if and only if \hyperref[rhy]{\rhy} is solvable. When solutions of \hyperref[rhx]{\rhx} and \hyperref[rhy]{\rhy} exist, they are unique and connected by \eqref{X}.
\end{lemma}

As typical in the steepest descent analysis of Riemann-Hilbert problems, we ignore the jump of $\boldsymbol X$ on $\Gamma_\mu\cup\Gamma_\sigma$ and look for the following approximation to $\boldsymbol X$ (\rhn):

\begin{itemize}
\label{rhn}
\item[(a)] $\boldsymbol N$ is analytic in $\CP\setminus(\Delta_\mu\cup\Delta_\sigma)$ and
\[
\lim_{z\to\infty} {\boldsymbol N}(z)\ \diag\left(z^{-n},z^{-m-1},z^{n+m+1}\right) = \boldsymbol I;
\]
\item[(b)] $\boldsymbol N$ has continuous traces on $\Delta_\mu^\circ\cup\Delta_\sigma^\circ$ that satisfy
\[
\boldsymbol N_+ = \boldsymbol N_- \mathsf T_\nu \left(\begin{matrix} 0 & \rho_\nu/w_\nu^+ \\ -w_\nu^+/\rho_\nu & 0 \end{matrix}\right) \quad \text{on} \quad \Delta_\nu^\circ, \quad \nu\in\{\mu,\sigma\};
\]
\item[(c)] $\boldsymbol N$ satisfies \hyperref[rhy]{\rhy}(c).
\end{itemize}

Let $\Phi_{m,n}$ be as defined before \eqref{normalization}, which are rational functions on the same surface $\RS$. Denote by $\Upsilon_k$, $k\in\{0,1,2\}$, a rational functions on $\RS$ with the divisor $\infty^{(0)}-\infty^{(k)}$, normalized as in \eqref{normalization}. Clearly, $\Upsilon_0\equiv1$, $\Phi_{m+1,n}\Upsilon_1=\Phi_{m+2,n-1}$ and $\Phi_{m+1,n}\Upsilon_2=\Phi_{m+1,n-1}$. Further, let $S:=S_c$ be the function granted by Proposition~\ref{prop:Szego}, again with respect to $\RS$. Define the constants $\gamma_{m+1,n}^{(k)}$ by
\begin{equation}
\label{constants}
\lim_{z\to\infty}z^{i(k)}\gamma_{m+1,n}^{(k)}\Phi_{m+1,n}^{(k)}(z)S^{(k)}(z)\Upsilon_k^{(k)}(z) =1,
\end{equation}
where $i(0)=-n$, $i(1)=-m-2$, and $i(2)=n+m$. Then the following lemma holds.

\begin{lemma}
\label{lem:rhn}
A solution of \hyperref[rhn]{\rhn} is given by $\boldsymbol N =  \boldsymbol C \boldsymbol M\boldsymbol D$, where
\[
\left\{
\begin{array}{lll}
\boldsymbol C & := & \diag\left(\begin{matrix} \gamma_{m+1,n}^{(0)} & \gamma_{m+1,n}^{(1)} &  \gamma_{m+1,n}^{(2)} \end{matrix} \right), \medskip \\
\boldsymbol M &:=& \left(\begin{matrix}
S^{(0)} & S^{(1)}/w_\sigma & S^{(2)}/w_\mu \medskip \\
S^{(0)}\Upsilon_1^{(0)} & S^{(1)}\Upsilon_1^{(1)}/w_\sigma & S^{(2)}\Upsilon_1^{(2)}/w_\mu \medskip \\
S^{(0)}\Upsilon_2^{(0)} & S^{(1)}\Upsilon_2^{(1)}/w_\sigma & S^{(2)}\Upsilon_2^{(2)}/w_\mu
\end{matrix} \right). \medskip \\
\boldsymbol D & := & \diag\left(\begin{matrix} \Phi_{m+1,n}^{(0)} & \Phi_{m+1,n}^{(1)} & \Phi_{m+1,n}^{(2)} \end{matrix} \right),
\end{array}
\right.
\]
\end{lemma}
\begin{proof}
Since $S$ is non-vanishing in the domain of holomorphy, the functions $w_\nu$ have simple poles at infinity, and the divisors of $\Phi_{m+1,n}$ and $\Upsilon_i$ are explicitly known, it is trivial to check that $\boldsymbol {CMD}$ satisfies \hyperref[rhn]{\rhn}(a). \hyperref[rhn]{\rhn}(b) follows easily from \eqref{S-jump} and the fact that $\Phi^{(1)\pm}=\Phi^{(0)\mp}$ on $\Delta_\sigma$ and $\Phi^{(1)\pm}=\Phi^{(2)\mp}$ on $\Delta_\mu$ for any rational function $\Phi$ on $\RS$. Finally, \hyperref[rhn]{\rhn}(c) is the consequence of the boundedness of $\Phi_{m,n}$ and $S$ around the endpoints of $\Delta_\mu\cup\Delta_\sigma$ and the choice of $w_\sigma$ and $w_\mu$.
\end{proof}

It can be readily checked that $\det(\boldsymbol N)$ is a holomorphic function in $\overline\CP\setminus (\Delta_\mu\cup\Delta_\sigma)$ and $\det(\boldsymbol N)(\infty)=1$. In fact, it has no jumps across $\Delta_\mu^\circ\cup\Delta_\sigma^\circ$ and since it is either bounded or behaves like $\mathcal O(|z-e|^{-1/2})$ near endpoints of $\Delta_\mu\cup\Delta_\sigma$, those points are in fact removable singularities. Therefore $\det(\boldsymbol N)$ is a bounded entire function. That is, $\det(\boldsymbol N)\equiv1$ as follows from the normalization at infinity. It also follows from \eqref{normalization} that $\det(\boldsymbol D)  \equiv 1$. In particular, this means that $\det(\boldsymbol M)$ is constant and non-zero in $\overline\CP$.

To take care of the jumps of $\boldsymbol X$ on $\Gamma_\mu\cup\Gamma_\sigma$, consider the following Riemann-Hilbert Problem (\rhz):
\begin{itemize}
\label{rhz}
\item[(a)] $\boldsymbol Z$ is a holomorphic matrix function in $\overline\CP\setminus(\Gamma_\mu\cup\Gamma_\sigma)$ and $\boldsymbol Z(\infty)=\boldsymbol I$;
\item[(b)] $\boldsymbol Z$ has continuous traces on $\Gamma_\mu\cup\Gamma_\sigma$ that satisfy
\[
\boldsymbol Z_+ = \boldsymbol Z_- (\boldsymbol {MD})\mathsf T_\nu \left(\begin{matrix} 1 & 0 \\ w_\nu/\rho_\nu & 1 \end{matrix}\right)(\boldsymbol {MD})^{-1} \quad \text{on} \quad \Gamma_\nu, \quad \nu\in\{\mu,\sigma\}.
\]
\end{itemize}
Then the following lemma takes place.
\begin{lemma}
\label{lem:rhz}
The solution of \hyperref[rhz]{\rhz} exists for all $n$ large enough and satisfies
\begin{equation}
\label{Z}
\boldsymbol Z=\boldsymbol I +\boldsymbol {\mathcal O}\big( C_{\mu,\sigma}^{-n}\big)
\end{equation}
for some constant $C_{\mu,\sigma}>1$, where $\boldsymbol {\mathcal O}(\cdot)$ holds uniformly in $\overline\CP$.
\end{lemma}
\begin{proof}
The jump matrix for $\boldsymbol Z$ on $\Gamma_\sigma$ is equal to
\begin{equation}
\label{geom1}
\boldsymbol I + \frac{ w_\sigma }{\rho_\sigma } \frac{\Phi_{m+1,n}^{(1)}}{\Phi_{m+1,n}^{(0)}} \boldsymbol M\boldsymbol E_{21}\boldsymbol M^{-1}.
\end{equation}
where $\boldsymbol E_{ij}$ is the matrix with all zero entries except for $(i,j)$-th, which is $1$. Since $\boldsymbol M$ is fixed and has constant determinant, it follows from \eqref{logPhi}, definition of $D_{\sigma,\frac n{n+m+1}}^+$ in \eqref{domains}, and the choice of $\Gamma_\sigma$ that the jump of $\boldsymbol Z$ on $\Gamma_\sigma$ is of the form $\boldsymbol I +\boldsymbol {\mathcal O}\big( C_{\mu,\sigma}^{-n}\big)$ for some $C_{\mu,\sigma}>1$. Similarly, we have that the jump $\boldsymbol Z$ on $\Gamma_\mu$ is equal to
\begin{equation}
\label{geom2}
\boldsymbol I + \frac{w_\mu}{\rho_\mu} \frac{\Phi_{m+1,n}^{(2)}}{\Phi_{m+1,n}^{(1)}} \boldsymbol M\boldsymbol E_{32}\boldsymbol M^{-1},
\end{equation}
which is also of the form $\boldsymbol I +\boldsymbol {\mathcal O}\big( C_{\mu,\sigma}^{-n}\big)$ for some properly adjusted $C_{\mu,\sigma}>1$ by \eqref{estPhi21}. The conclusion of the lemma follows now from the same argument as in \cite[Corollary~7.108]{Deift} with another adjustment of $C_{\mu,\sigma}$.
\end{proof}

Finally, let $\boldsymbol Z$ be a solution of \hyperref[rhz]{\rhz} granted by Lemma~\ref{lem:rhz} and $\boldsymbol N=\boldsymbol{CMD}$ be as in Lemma~\ref{lem:rhn}. Then it can be easily checked that $\boldsymbol X = \boldsymbol{CZMD}$ solves \hyperref[rhx]{\rhx} and therefore
\begin{equation}
\label{solution}
\boldsymbol Y = \boldsymbol{CZMD} \left\{
\begin{array}{ll}
\mathsf T_\nu\left( \begin{matrix} 1 & 0  \\ w_\mu/\rho_\nu & 1\end{matrix} \right), & \text{in} \quad \Omega_\nu, \quad \nu\in\{\mu,\sigma\}, \medskip \\
\boldsymbol I, & \text{otherwise}.
\end{array}
\right.
\end{equation}
solves \hyperref[rhy]{\rhy}.

\subsection{Non-Linear Steepest Descent Analysis in the Case $\Delta_{\sigma,c}\cap\partial D_{\sigma,c}^-\neq\varnothing$}
\label{ssec:local}

In the case $\Delta_{\sigma,c}\cap\partial D_{\sigma,c}^-\neq\varnothing$ there no longer exists a Jordan curve $\Gamma_\sigma\subset D_{\sigma,c}^+$ encircling $\Delta_{\sigma,c}$, which prevents us from carrying out the estimate \eqref{geom1}.
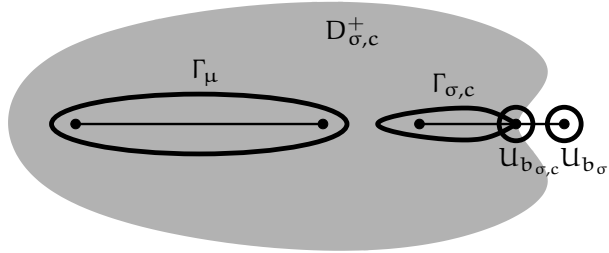
\begin{figure}[ht!]
\centering
\includegraphics[scale=.5]{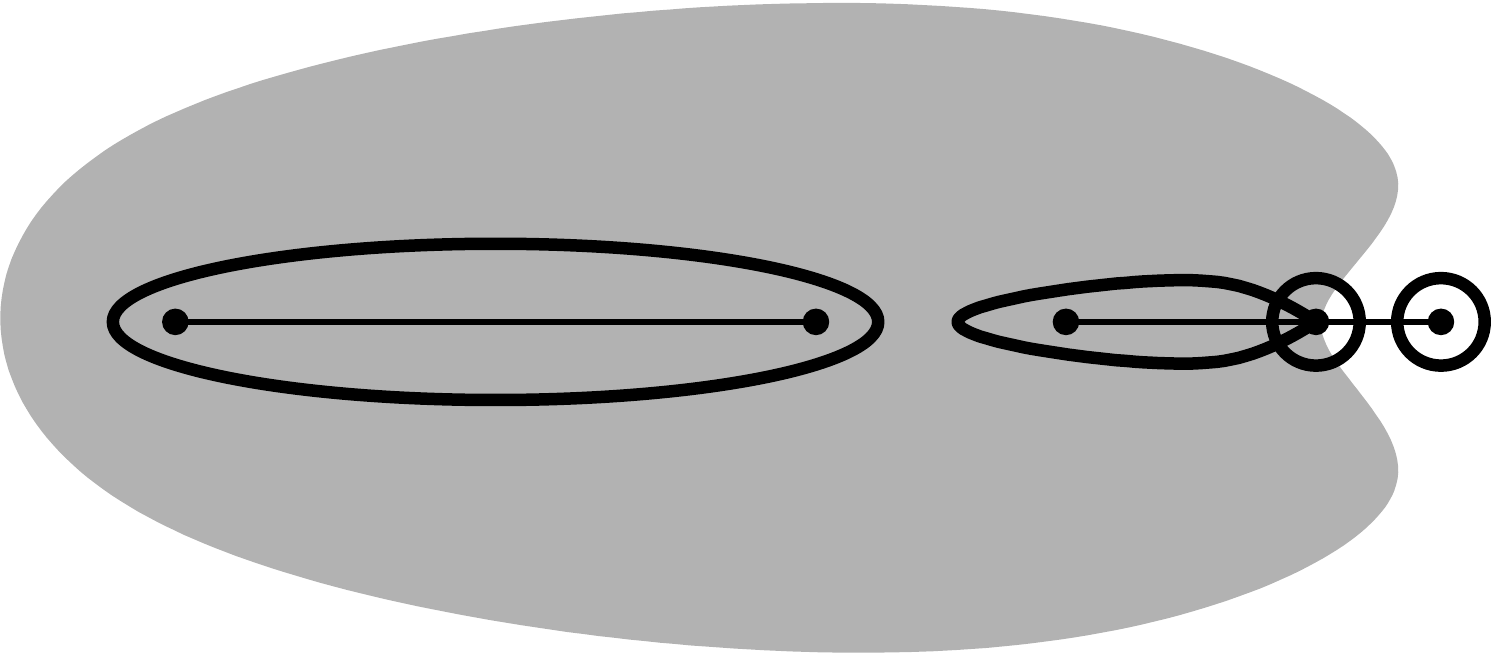}
\begin{picture}(0,0)
\put(-100,80){$D_{\sigma,c}^+$}
\put(-150,65){$\Gamma_\mu$}
\put(-60,60){$\Gamma_{\sigma,c}$}
\put(-35,32){$U_{b_{\sigma,c}}$}
\put(-12,32){$U_{b_\sigma}$}
\end{picture}
\caption{\small Contours $\Gamma_\mu$, $\Gamma_{\sigma,c}$, and the disks $U_{b_{\sigma,c}}$, $U_{b_\sigma}$.}
\label{fig:lens}
\end{figure}
Hence, we shall require the Jordan curve $\Gamma_{\sigma,c}$ to encircle $\Delta_{\sigma,c}$ except for the point $b_{\sigma,c}$, which they have in common. Moreover, given disjoint disks $U_{b_{\sigma,c}}$ and $U_{b_\sigma}$ centered at $b_{\sigma,c}$ and $b_\sigma$, respectively, (unless $b_{\sigma,c}=b_\sigma$ in which case these disks coincide) we also require that $\Gamma_{\sigma,c}\setminus U_{b_{\sigma,c}}\subset D_{\sigma,c}^+$, see Figure~\ref{fig:lens}. To slightly alleviate the notation, let us set
\[
b_{m+1,n}:=b_{\sigma,\frac n{n+m+1}}, \quad \Delta_{m+1,n}:=\Delta_{\sigma,\frac n{n+m+1}}, \text{and} \quad \Gamma_{m+1,n}:=\Gamma_{\sigma,\frac n{n+m+1}},
\]
where the curves $\Gamma_{m+1,n}$ are selected analogously to $\Gamma_{\sigma,c}$ with the requirement that $\Gamma_{m+1,n}\to\Gamma_{\sigma,c}$ as $n\to\infty$ in Hausdorff metric (recall that $b_{m+1,n}\to b_{\sigma,c}$ as $n\to\infty$, see \eqref{convVs}). We take $\Gamma_\mu$ to be a Jordan curve encircling $\Delta_\mu$, which is disjoint from all $\Gamma_{m+1,n}$. As before, we assume that $\rho_\mu$ is holomorphic across $\Gamma_\mu$ and $\rho_\nu$ is holomorphic across each $\Gamma_{m+1,n}$. We continue to denote by $\Omega_\mu$ and $\Omega_{m+1,n}$ the domains bounded by $\Gamma_\mu\cup\Delta_\mu$ and $\Gamma_{m+1,n}\cup\Delta_{m+1,n}$, respectively.

In what follows, we shall often refer back to Riemann-Hilbert problems formulated in Section~\ref{ssec:no-local}. For each such reference it is understood that when $\Delta_\sigma$, $\Gamma_\sigma$, and $\Omega_\sigma$ occur, they should be replaced by $\Delta_{m+1,n}$, $\Gamma_{m+1,n}$, and $\Omega_{m+1,n}$.

Define $\boldsymbol X$ by \eqref{X}. Then $\boldsymbol X$ satisfies \hyperref[rhx]{\rhx}(a,c) and \hyperref[rhx]{\rhx}(b) with an additional jump
\[
\boldsymbol X_+ = \boldsymbol X_- \mathsf T_\sigma \left(\begin{matrix} 1 & \rho_\sigma/w_\sigma^+ \\ 0 & 1 \end{matrix}\right) \quad \text{on} \quad \Delta_\sigma^\circ\setminus\Delta_{m+1,n}.
\]
Clearly, Lemma~\ref{lem:rhs} remains valid.

Define $\boldsymbol D$ as in Lemma~\ref{lem:rhn}, where $\Phi_{m+1,n}$ is a rational function on $\RS_{m+1,n}$ with the divisor $(n+m+1)\infty^{(2)}-n\infty^{(0)}-(m+1)\infty^{(1)}$ and normalized as in \eqref{normalization}. Let $S_{m+1,n}$ be the function on $\RS_{m+1,n}$ granted by Proposition~\ref{prop:Szego} applied with $c=\frac n{n+m+1}$ and $\Upsilon_k:=\Upsilon_{k;m+1,n}$, $k\in\{0,1,2\}$, be the rational function on $\RS_{m+1,n}$ with the divisor $\infty^{(0)}-\infty^{(k)}$ and the normalization as in \eqref{normalization}. Define matrices $\boldsymbol M$ and $\boldsymbol C$ as in Lemma~\ref{lem:rhn} using the above functions. Then $\boldsymbol N=\boldsymbol{CMD}$ again solves \hyperref[rhn]{\rhn} and it is still true that $\det(\boldsymbol N)\equiv1$. Therefore $\det(\boldsymbol M)$ is a constant, but in this case it might depend on $(m,n)$. However, observe that an analogous matrix $\boldsymbol M=\boldsymbol M_c$ can be defined on the ``limiting'' surface $\RS_c$ as well. Moreover, it was shown in \cite[Section~7]{uYa} that
\begin{equation}
\label{convFuns}
S_{m+1,n} \to S_c \quad \text{and} \quad \Upsilon_{k;m+1,n} \to \Upsilon_{k;c}
\end{equation}
uniformly on $\RS_{c,\delta}$ for any $\delta>0$, where $\RS_{c,\delta}$ is obtained from $\RS_c$ be removing circular neighborhoods of radius $\delta$ around each branch point of $\RS_c$ and functions $S_{m+1,n}$ and $\Upsilon_{k;m+1,n}$ are carried over to $\RS_{c,\delta}$ with the help of natural projections. Hence, $\det(\boldsymbol M_{m+1,n})\to\det(\boldsymbol M_c)$, in particular, the determinants $\det(\boldsymbol M)$ are uniformly bounded away from zero and infinity with $m$ and $n$.

Let $\boldsymbol P_b$, $b\in\{b_{\sigma,c},b_\sigma\}$, be a matrix-valued function that solves \hyperref[rhx]{\rhx} inside of $U_b$ and satisfies
 \begin{equation}
 \label{Pb}
 \boldsymbol P_b = \boldsymbol M \big( \boldsymbol I + \boldsymbol {\mathcal O}\big(\varepsilon_{m+1,n}\big) \big)\boldsymbol D
 \end{equation}
 uniformly on $\partial U_b$, where $0<\varepsilon_{m+1,n}\to0$ as $n\to\infty$. Such matrices do exist. Indeed, when $b_{\sigma,c}\neq b_\sigma$, one can easily check that
 \[
 \boldsymbol P_{b_\sigma} = \boldsymbol M\mathsf T_\sigma \left(\begin{matrix} 1 & \mathcal C_\sigma \Phi^{(0)}_{m+1,n}/\Phi^{(1)}_{m+1,n}\\ 0 & 1 \end{matrix}\right)\boldsymbol D,
 \]
 where $\mathcal C_\sigma(z):=\frac1{2\pi\mathrm i}\int_{\Delta_\sigma}\frac{\rho_\sigma(x)}{x-z}\frac{\mathrm d\sigma(x)}{w_\sigma^+(x)}$. As the construction of $\boldsymbol P_{b_{\sigma,c}}$ is quite long and is absolutely identical to the one in \cite[Sections~9.4 and~9.5]{uYa}, we omit it here. Let us just mention that it is based on the model Riemann-Hilbert problem associated with solutions Panlev\'e XXXIV equation \cite[Section~4.2]{uYa}, see also \cite{IKOs08,IKOs09}.

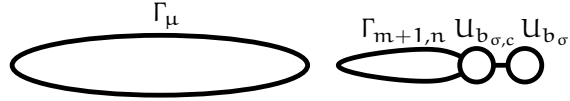
\begin{figure}[ht!]
\centering
\includegraphics[scale=.5]{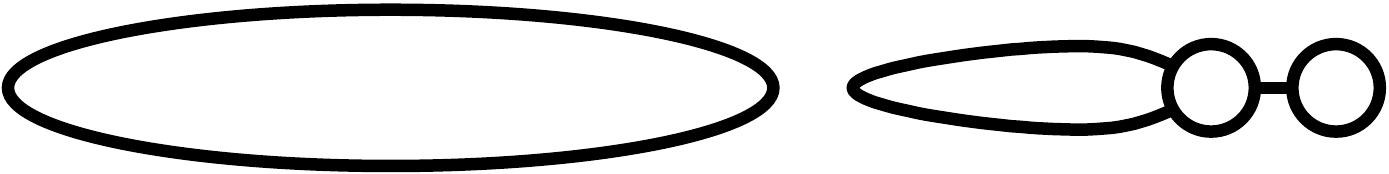}
\begin{picture}(0,0)
\put(-150,29){$\Gamma_\mu$}
\put(-72,23){$\Gamma_{m+1,n}$}
\put(-37,23){$U_{b_{\sigma,c}}$}
\put(-12,23){$U_{b_\sigma}$}
\end{picture}
\caption{\small Lens $\Sigma_{\boldsymbol Z}$ consisting of the curve $\Gamma_\mu$, the arc $\Gamma_{m+1,n}\setminus U_{b_{\sigma,c}}$, the interval $\Delta_\sigma\setminus(\Delta_{\sigma,c}\cup U_{b_{\sigma,c}}\cup U_{b_\sigma})$, and the circles $\partial U_{b_{\sigma,c}}$, $\partial U_{b_\sigma}$.}
\label{fig:lensZ}
\end{figure}

Let the contour $\Sigma_{\boldsymbol Z}$ be as depicted on Figure~\ref{fig:lensZ}. The last matrix-valued function needed to solve \hyperref[rhy]{\rhy} is described by the following Riemann-Hilbert problem:
\begin{itemize}
\label{rhz2}
\item[(a)] $\boldsymbol Z$ is a holomorphic matrix function in $\overline\CP\setminus\Sigma_{\boldsymbol Z}$ and $\boldsymbol Z(\infty)=\boldsymbol I$;
\item[(b)] $\boldsymbol Z$ has continuous traces on $\Sigma_{\boldsymbol Z}$ except perhaps at its branching points that satisfy
\[
\boldsymbol Z_+ = \boldsymbol Z_- (\boldsymbol {MD})\mathsf T_\nu \left(\begin{matrix} 1 & 0 \\ w_\nu/\rho_\nu & 1 \end{matrix}\right)(\boldsymbol {MD})^{-1}
\]
on $\Gamma_\mu$ when $\nu=\mu$ and $\Gamma_{m+1,n}\setminus U_{b_{\sigma,c}}$ when $\nu=\sigma$,
\[
\boldsymbol Z_+ = \boldsymbol Z_- (\boldsymbol {MD})\mathsf T_\sigma \left(\begin{matrix} 1 & \rho_\sigma/w_\sigma^+ \\ 0 & 1 \end{matrix}\right)(\boldsymbol {MD})^{-1}
\]
on $\Delta_\sigma\setminus(\Delta_{\sigma,c}\cup U_{b_{\sigma,c}}\cup U_{b_\sigma})$, and  $\boldsymbol Z_+ = \boldsymbol Z_-\boldsymbol P_b(\boldsymbol {MD})^{-1}$, on $\partial U_b$, $b\in\{b_{\sigma,c},b_\sigma\}$.
\end{itemize}
Exactly as in the previous case, the following lemma holds.
\begin{lemma}
\label{lem:rhz2}
The solution of \hyperref[rhz2]{\rhz} exists for all $n$ large enough and satisfies
\begin{equation}
\label{Z2}
\boldsymbol Z=\boldsymbol I +\boldsymbol {\mathcal O}\big( \varepsilon_{m+1,n} \big)
\end{equation}
where $\boldsymbol {\mathcal O}(\cdot)$ holds uniformly in $\overline\CP$ and $\varepsilon_{m+1,n}$ are the constants from \eqref{Pb}.
\end{lemma}
\begin{proof}
Recall that the determinants $\det(\boldsymbol M)$ are identically constant as functions of $z$ and that these constants are uniformly separated from zero and infinity with $m$ and $n$. Hence, the estimates of the size of the jumps on $\Gamma_\mu$ and $\Gamma_{m+1,n}\setminus U_{b_{\sigma,c}}$ are absolutely analogous to \eqref{geom1} and \eqref{geom2}. The jumps in this case are geometrically close to the identity where the constant of proportionality depends on how close $\Gamma_{m+1,n}\setminus U_{b_{\sigma,c}}$ is to $\partial D_{\sigma,c}$ (the latter sets are uniformly separated from each other by our construction of $\Gamma_{m+1,n}$). The jump on $\Delta_\sigma\setminus(\Delta_{\sigma,c}\cup U_{b_{\sigma,c}}\cup U_{b_\sigma})$ is also geometrically close to the identity since this interval belongs to $D_{\sigma,c}^-$ where $|\Phi_{m+1,n}^{(1)}|>|\Phi_{m+1,n}^{(0)}|$. Finally, we see that the jump on $\partial U_b$, $b\in\{b_{\sigma,c},b_\sigma\}$, is $\boldsymbol {\mathcal O}\big( \varepsilon_{m+1,n} \big)$ close to the identity by \eqref{Pb} and the normality of $\boldsymbol M$, see \eqref{convFuns}. The existence of $\boldsymbol Z$ again follows from \cite[Corollary~7.108]{Deift}. The size of the error is proportional to $ \varepsilon_{m+1,n}$ as the latter is of order $1/n$ at best, see \cite[Sections~9.4 and~9.5]{uYa}.
\end{proof}

Altogether, the solution of \hyperref[rhx]{\rhx} is given by
\begin{equation}
\label{solution2}
\boldsymbol X = \boldsymbol{CZ} \left\{
\begin{array}{ll}
\boldsymbol P_b, & \text{in} \quad U_b, \quad b\in\{b_{\sigma,c},b_\sigma\}, \medskip \\
\boldsymbol{MD}, & \text{otherwise},
\end{array}
\right.
\end{equation}
and then the solution of \hyperref[rhy]{\rhy} is obtained by inverting \eqref{X}.

\subsection{Asymptotic Analysis}

Below we write $S_{m+1,n}$ and $\Upsilon_{k;m+1,n}$ irrespectively of whether we are in the case of Section~\ref{ssec:no-local} or Section~\ref{ssec:local}. Write $\boldsymbol Z=\boldsymbol I + \big[\upsilon_{m+1,n}^{(i,j)}\big]_{i,j=1}^3$, where we know from Lemmas~\ref{lem:rhz} and~\ref{lem:rhz2} that
\begin{equation}
\label{final}
|\upsilon_{m+1,n}^{(i,j)}| =\mathcal O\big(C_{\mu,\sigma}^{-n}\big) \quad \text{or} \quad |\upsilon_{m+1,n}^{(i,j)}| =\mathcal O\big(\varepsilon_{m+1,n}\big)
\end{equation}
uniformly in $\overline\CP$ depending on the considered case ($\upsilon_{m+1,n}^{(i,j)}(\infty)=0$ as $\boldsymbol{Z}(\infty)=\boldsymbol{I}$). Given any closed set $K\subset\overline\CP\setminus\Delta_\sigma$, choose $\Omega_\sigma$ or $\Omega_{m+1,n}\cup\bigcup_b U_b$ so that $K$ belongs to the complement of its closure. Then we get from \eqref{solution} and \eqref{solution2} that $\boldsymbol Y=\boldsymbol{CZMD}$ on $K$ and therefore
\[
[\boldsymbol Y]_{11} = \gamma_{m+1,n}^{(0)}\Phi_{m+1,n}^{(0)}S_{m+1,n}^{(0)}\left( 1+\upsilon_{m+1,n}^{(1,1)} + \upsilon_{m+1,n}^{(1,2)}\Upsilon_{1;m+1,n}^{(0)} + \upsilon_{m+1,n}^{(1,3)}\Upsilon_{2;m+1,n}^{(0)} \right)
\]
on $K$ regardless whether it intersects $\Omega_\mu$ or not. The first relation in \eqref{asymptotics1} now follows from \eqref{Y}, \eqref{convFuns}, and \eqref{final}. On the other hand, we get from \eqref{solution} and \eqref{solution2} that
\begin{multline*}
[\boldsymbol Y]_{11} = \gamma_{m+1,n}^{(0)}\Phi_{m+1,n}^{(0)\pm}S_{m+1,n}^{(0)\pm}\left( 1+\upsilon_{m+1,n}^{(1,1)} + \upsilon_{m+1,n}^{(1,2)}\Upsilon_{1;m+1,n}^{(0)\pm} + \upsilon_{m+1,n}^{(1,3)}\Upsilon_{2;m+1,n}^{(0)\pm}  \right) + \\
\gamma_{m+1,n}^{(0)}\Phi_{m+1,n}^{(1)\pm}\frac{w_\sigma^\pm}{w_{\sigma,c}^\pm\rho_\sigma}S_{m+1,n}^{(1)\pm}\left( 1+\upsilon_{m+1,n}^{(1,1)} + \upsilon_{m+1,n}^{(1,2)}\Upsilon_{1;m+1,n}^{(1)\pm} + \upsilon_{m+1,n}^{(1,3)}\Upsilon_{2;m+1,n}^{(1)\pm} \right)
\end{multline*}
on $\Delta_\sigma$ or $\Delta_{\sigma,c}\setminus U_{b_{\sigma,c}}$, depending on the considered case. The second relation in \eqref{asymptotics1} now follows from \eqref{Y}, \eqref{S-jump}, \eqref{convFuns}, \eqref{final}, and the fact that $\Upsilon_k^{(0)\pm}=\Upsilon_k^{(1)\mp}$ on $\Delta_{\sigma,c}^\circ$.

Now, let $K\subset\overline\CP\setminus(\Delta_\mu\cup\Delta_\sigma)$. Adjust the set $\Omega_\mu\cup\Omega_\sigma$ or $\Omega_\mu\cup\Omega_{m+1,n}\cup\bigcup_b U_b$ if necessary so that $K$ belongs to the complement of its closure. Then $\boldsymbol Y=\boldsymbol{CZMD}$ on $K$ and therefore
\[
[\boldsymbol Y]_{12} = \frac{\gamma_{m+1,n}^{(0)}}{w_{m+1,n}}\Phi_{m+1,n}^{(1)}S_{m+1,n}^{(1)}\left( 1+\upsilon_{m+1,n}^{(1,1)} + \upsilon_{m+1,n}^{(1,2)}\Upsilon_{1;m+1,n}^{(1)} + \upsilon_{m+1,n}^{(1,3)}\Upsilon_{2;m+1,n}^{(1)} \right).
\]
Even though $\Upsilon_{1;m+1,n}^{(1)}$ has a pole at infinity, the product $\upsilon_{m+1,n}^{(1,2)}\Upsilon_{1;m+1,n}^{(1)}$ is finite satisfies \eqref{final} by \eqref{convFuns} and the maximum modulus principle. As $w_{m+1,n}^{-1}\to w_\sigma^{-1}$ locally uniformly in $\overline\CP\setminus\{a_\sigma,b_{\sigma,c}\}$, the first relation in \eqref{asymptotics2} follows from \eqref{Y}, \eqref{convFuns}, and \eqref{final}. Observe also that the last equality essentially does not change on $\Delta_\sigma$ or $\Delta_{\sigma,c}\setminus U_{b_{\sigma,c}}$, i.e., we simply need to replace the functions by their boundary values. This yields the second formula in \eqref{asymptotics2}. Finally, we get that
\begin{multline*}
[\boldsymbol Y]_{12} = \gamma_{m+1,n}^{(0)}\Phi_{m+1,n}^{(1)\pm}\frac{S_{m+1,n}^{(1)\pm}}{w_{\sigma,c}}\left( 1+\upsilon_{m+1,n}^{(1,1)} + \upsilon_{m+1,n}^{(1,2)}\Upsilon_{1;m+1,n}^{(1)\pm} + \upsilon_{m+1,n}^{(1,3)}\Upsilon_{2;m+1,n}^{(1)\pm} \right) + \\
\gamma_{m+1,n}^{(0)}\Phi_{m+1,n}^{(2)\pm}\frac{S_{m+1,n}^{(2)\pm}}{\rho_\mu}\left( 1+\upsilon_{m+1,n}^{(1,1)} + \upsilon_{m+1,n}^{(1,2)}\Upsilon_{1;m+1,n}^{(2)\pm} + \upsilon_{m+1,n}^{(1,3)}\Upsilon_{2;m+1,n}^{(2)\pm} \right)
\end{multline*}
on $\Delta_\mu$. The third relation in \eqref{asymptotics2} now follow from \eqref{Y}, \eqref{S-jump}, \eqref{convFuns}, and \eqref{final}. This finishes the proof of Theorem~\ref{thm:Frobenius} since the uniformity of the estimates in the case $b_{\sigma,c}\not\in\partial D_{\sigma,c}^-$ follows from the fact $S_{m+1,n}=S_c$ and $\Upsilon_{k;m+1,n}=\Upsilon_k$ for all $n$ large enough and therefore we do not need to use \eqref{convFuns}.

\end{document}